\documentclass{amsart}
\usepackage{amsmath}

\usepackage{amssymb}
\usepackage{amscd}
\usepackage{hyperref}
\hypersetup{
    bookmarks=true,         
    unicode=false,          
    pdftoolbar=true,        
    pdfmenubar=true,        
    pdffitwindow=false,     
    pdfstartview={FitH},    
    pdftitle={My title},    
    pdfauthor={Author},     
    pdfsubject={Subject},   
    pdfcreator={Creator},   
    pdfproducer={Producer}, 
    pdfkeywords={keyword1} {key2} {key3}, 
    pdfnewwindow=true,      
    colorlinks=true,       
    linkcolor=black,          
    citecolor=black,        
    filecolor=black,      
    urlcolor=black           
}
\usepackage{longtable}
\newtheorem{theorem}{Theorem}[section]
\newtheorem{lemma}[theorem]{Lemma}
\newtheorem{proposition}[theorem]{Proposition}
\newtheorem{corollary}[theorem]{Corollary}

\theoremstyle{definition}
\newtheorem{definition}[theorem]{Definition}

\newtheorem{example}[theorem]{Example}
\newtheorem{examples}[theorem]{Examples}
\newtheorem{definitions and remarks}[theorem]{Definitions and Remarks}

\theoremstyle{remark}
\newtheorem{remark}[theorem]{Remark}
\newtheorem{remarks}[theorem]{Remarks}

\numberwithin{equation}{section}

\newcommand{\pinv}{\kappa\text{-}inv}

\newcommand{\inv}{\mathrm{inv}}

\newcommand{\Sing}{\mathrm{Sing}\,}
\newcommand{\Supp}{\mathrm{Supp}\,}
\newcommand{\supp}{\mathrm{supp}}
\newcommand{\cosupp}{\mathrm{cosupp}\,}
\newcommand{\mon}{\mathrm{mon}}
\newcommand{\ord}{\mathrm{ord}}

\newcommand{\length}{\mathrm{length}\,}

\newcommand{\lcm}{\mathrm{lcm}}

\newcommand{\al}{{\alpha}}

\newcommand{\ka}{{\kappa}}
\newcommand{\la}{{\lambda}}
\newcommand{\La}{{\Lambda}}
\newcommand{\Om}{{\Omega}}

\newcommand{\s}{{\sigma}}

\newcommand{\io}{{\iota}}

\newcommand{\IN}{{\mathbb N}}

\newcommand{\IK}{{\mathbb K}}

\newcommand{\cB}{{\mathcal B}}

\newcommand{\cE}{{\mathcal E}}

\newcommand{\cG}{{\mathcal G}}
\newcommand{\cH}{{\mathcal H}}
\newcommand{\cI}{{\mathcal I}}
\newcommand{\cJ}{{\mathcal J}}
\newcommand{\cK}{{\mathcal K}}
\newcommand{\cM}{{\mathcal M}}
\newcommand{\cN}{{\mathcal N}}
\newcommand{\cO}{{\mathcal O}}

\newcommand{\cR}{{\mathcal R}}
\newcommand{\cS}{{\mathcal S}}
\newcommand{\cV}{{\mathcal V}}

\newcommand{\fm}{{\mathfrak m}}
\newcommand{\fn}{{\mathfrak n}}

\newcommand{\tD}{{\widetilde D}}

\newcommand{\wA}{{\widehat A}}

\newcommand{\ucB}{\underline{\cB}}

\newcommand{\ucG}{\underline{\cG}}
\newcommand{\ucH}{\underline{\cH}}
\newcommand{\ucI}{\underline{\cI}}
\newcommand{\ucJ}{\underline{\cJ}}

\newcommand{\llbracket}{{[\![}}
\newcommand{\rrbracket}{{]\!]}}

\begin{document}
\title[Desingularization preserving stable simple normal crossings]{Desingularization preserving stable simple normal crossings}

\author[E.~Bierstone]{Edward Bierstone}
\author[F.~Vera Pacheco]{Franklin Vera Pacheco}
\address{University of Toronto, Department of Mathematics, 40 St. George Street,
Toronto, ON, Canada M5S 2E4, and the Fields Institute, 222 College Street, Toronto, ON, Canada M5T 3J1}
\email[E.~Bierstone]{bierston@math.toronto.edu}
\email[F.~Vera Pacheco]{franklin.vp@gmail.com}
\thanks{Research supported in part by NSERC grants OGP0009070 and
MRS342058.}



\begin{abstract}
The subject is partial resolution of singularities. Given an algebraic variety $X$ (not necessarily equidimensional) 
in characteristic zero (or, more generally, a pair $(X,D)$, where $D$ is a divisor on $X$), we construct a functorial
desingularization of all but \emph{stable simple normal crossings (stable-snc)} singularities, 
by smooth blowings-up that preserve
such singularities. A variety has \emph{stable simple normal crossings} at a point if, locally, its irreducible components
are smooth and tranverse in some smooth embedding variety. We also show that our main assertion is false for more
general simple normal crossings singularities.
\end{abstract}
\date{\today}
\maketitle
\setcounter{tocdepth}{1}
\tableofcontents

\section{Introduction}\label{sec:intro}
The subject of this article is partial resolution of singularities. Let $X$ denote a
(reduced) algebraic variety $X$ over a field of characteristic zero and let $D$ denote a $\mathbb{Q}$-Weil divisor
on $X$. Our main result (see Theorem \ref{thm:mainpairs}) asserts that we
can resolve all but \emph{stable simple normal crossings} singularities of $(X,D)$ by a finite sequence of
blowings-up, each of which is an isomorphism over the stable simple normal crossings points of its target.
See Definitions \ref{def:stablesnc}, \ref{def:stablesncpairs}, and Lemma \ref{lem:stablesnc}. The theorem
is functorial (Remarks \ref{rem:main}) and is obtained by an algorithm. Theorem \ref{thm:mainpairs} is false for more
general normal crossings singularities; see Example \ref{ex:counterex} (of course,
a weaker desingularization
result may hold in this case). We do not assume that $X$ is equidimensional;
in particular, we do not
define simple normal crossings singularities in a way that they are necessarily hypersurface singularities, as in \cite{BV}.
Our main theorem generalizes \cite{BV}; simple normal crossings hypersurface singularities are necessarily
stable.

For background and motivation of the problem,
see \cite{BMmin}, \cite{BV} and \cite{Kolog}. Our proof follows the philosophy of \cite{BMmin} that the desingularization invariant of \cite{BMinv} and \cite{BMfunct}
can be used together with natural geometric
information to compute local normal forms of singularities.

\begin{definition}\label{def:stablesnc} 
A (reduced) algebraic variety $X$ has a \emph{stable simple normal crossings} (\emph{stable-snc}) 
singularity at a point $a$ (or $X$ is \emph{stable-snc at $a$}) if
the irreducible components $X^{(i)}$ of $X$ are smooth at $a$, and are transverse at $a$ in some smooth embedding
variety $Z$ of a neighbourhood of $a$ in $X$ (i.e., the sum of the codimensions in $Z$ of the
tangent spaces of the $X^{(i)}$ at $a$ equals the codimension of the intersection of the
tangent spaces).
\end{definition}

Note that $Z$ in the definition is necessarily a minimal local embedding variety. We say
 that $X$ has a
\emph{simple normal crossings} (\emph{snc}) singularity at $a$ if there is a smooth local 
embedding variety $Z$ at $a$ with a system
of regular coordinates in which each $X^{(i)}$ is a coordinate subspace. (This is
a more general notion than ``$X$ is locally isomorphic to a simple normal crossings divisor'',
often used as the definition.)

\begin{lemma}\label{lem:stablesnc}. Let $X$ denote an algebraic variety, and let $X^{(i)}$ denote the irreducible
components of $X$. Let $a \in X$. Assume that each $X^{(i)}$ is smooth at $a$. Then the following conditions
are equivalent:
\begin{enumerate}
\item $X$ is stable-snc at $a$.

\item If $Z$ is a smooth local embedding variety of $X$ at $a$, then $Z$ admits a system of regular coordinates
$(x_1,\ldots, x_p, z_1,\ldots,z_r,w_1,\ldots,w_s)$ at $a$, with respect to which each $X^{(i)}=(\{x_k = 0\}_{k\in I_i}, z_1=\ldots=z_r=0)$, 
for some partition $\cup_{i=1}^{m} I_i$ of $\{1,\ldots,p\}$.

\item $X$ is snc at $a$ and there is a smooth local embedding variety in which any two components $X^{(i)}$
are transverse at $a$.

\item The intersection of the $X^{(i)}$ is smooth (as a scheme) at $a$, and $X$ admits a smooth local embedding
variety $Z$ at $a$ in which the sum
of the codimensions of the $X^{(i)}$ equals the codimension of their
intersection. (See also \eqref{eq:codimstablesnc}.)

\item $X$ admits a smooth local embedding variety at $a$ in which the $X^{(i)}$ are smooth and in general position.
\end{enumerate}
\end{lemma}

\begin{remark}\label{rem:stablesnc}
It follows from Lemma \ref{lem:stablesnc} that, if $X$ is stable-snc at $a$, then the conditions (3)--(5) and the 
transversality property
of Definition \ref{def:stablesnc} are satisfied in any minimal embedding variety of $X$ at $a$.
\end{remark}

\begin{definition}\label{def:stablesncpairs}
Let $X$ denote a (reduced) algebraic variety and let $X^{(i)}$ denote the irreducible
components of $X$. Let $D$ denote a $\mathbb{Q}$-Weil divisor on $X$, i.e. $D$ is
a finite linear combination of reduced, irreducible subvarieties of $X$, each of codimension
one in any $X^{(i)}$ that contains it. We say that $(X,D)$ has (or is) \emph{stable simple normal crossings} (\emph{stable-snc}) 
at a point $a$ if there is a local embedding at $X \hookrightarrow Z$ 
at $a$, where $Z$ is smooth and admits a regular system of coordinates 
$(x_1,\ldots, x_p, y_1,\ldots, y_q,z_1,\ldots,z_r,w_1,\ldots,w_s)$ at $a$
in which
\begin{enumerate}
\item each $X^{(i)}:=(\{x_k = 0\}_{k\in I_i}, z_1=\ldots=z_r=0)$, for some partition $\cup_{i=1}^{m} I_i$ of $\{1,\ldots,p\}$;
\item
$D=\sum_{j=1}^{k}\alpha_j(y_j=0)|_{X}$ (locally at $a$), for some $\alpha_{j}\in\mathbb{Q}$. 
\end{enumerate}
We also say that the pair $(X,D)$ is \emph{stable-snc} if it is stable-snc at every point.
\end{definition}

It follows that, if $(X,D)$ is stable-snc at $a$, then any smooth local embedding variety at $a$ admits a
regular coordinate system as in Definition \ref{def:stablesncpairs}.

Observe that in Definition \ref{def:stablesncpairs} we do not assume \emph{a priori} that $D$ arises from the intersection with $X$ of a divisor on $Z$, though of course this property is satisfied if
$(X,D)$ is stable-snc.

\begin{example}\label{ex:snc-stablesnc}
 Consider $X:=(x=y=0)\cup(y=z=0)\cup(x=z=0)\subset\mathbb{A}_{x,y,z}^{3}$. Then $X$ is snc at the origin but not stable-snc. On the other hand, $Y:=(x=y=0)\cup(y=z=0)$ is stable-snc.
\end{example}

\begin{example}\label{ex:multiplicities}
If $X=(xy=0)\subset\mathbb{A}^3$ and $D=a_1D_1+a_2D_2$, where $D_1=(x=z=0)$ and $D_2=(y=z=0)$, then the 
pair $(X,D)$ is stable-snc if and only if $a_1=a_2$.
\end{example}

\begin{definition}\label{not:seq} \emph{Transform} of a pair $(X,D)$. Consider a sequence of blowings-up
\begin{equation}\label{eq:bluppairs}
X = X_0 \stackrel{\s_1}{\longleftarrow} X_1 \longleftarrow \cdots
\stackrel{\s_{t}}{\longleftarrow} X_{t}\,,
\end{equation}
where each $\s_{j+1}$ has smooth centre $C_j \subset X_j$. Write
$\tD_0:=D$ and, for each $j = 0,1,\ldots$\,, set
$\tD_{j+1} :=$ the birational transform of $\tD_j$ plus the \emph{exceptional divisor}
$\s_{j+1}^{-1}(C_j)$ of $\s_{j+1}$.
\end{definition}

\begin{theorem}\label{thm:mainpairs}
Let $X$ denote a (reduced) algebraic variety in characteristic zero and let $D$ denote a $\mathbb{Q}$-Weil 
divisor on $X$. Then there is a sequence of
blowings-up \eqref{eq:bluppairs} such that
\begin{enumerate}
\item $(X_t, \tD_t)$ has only stable-snc singularities;
\item\label{conditiontwo} each $\s_{j+1}$ is an isomorphism over the locus of stable-snc points of $(X_j, \tD_j)$.
\end{enumerate}
\end{theorem}

\begin{remarks}\label{rem:main}
(1) In the special case that $D=0$, each $\tD_j$ is the exceptional divisor of the morphism $\s_1\circ \cdots \circ \s_j$,
so that condition (1) of Theorem \ref{thm:mainpairs} is a stronger
assertion than ``$X_t$ is stable-snc''.
\smallskip

\noindent
(2) In the special case that $X$ is smooth, we say that $D$ is a  \emph{simple normal crossings} or \emph{snc} divisor on 
$X$ if $(X,D)$ is stable-snc (i.e., Definition \ref{def:stablesncpairs} is satisfied with $p=0$ at every point of X). This means
that the components of $D$ are smooth and intersect transversely. Theorem \ref{thm:mainpairs} in this case provides 
\emph{log resolution of singularities} of $D$ by a sequence of blowings-up \eqref{eq:bluppairs} such that each 
$\sigma_{j+1}$ is an isomorphism over the snc locus of $\tD_j$. This is proved in \cite[Thm.\,1.5]{BDV}. Earlier versions can be found in \cite{Sz}, \cite[Sect.\,12]{BMinv}, \cite{Kolog} and \cite[Thm.\,3.1]{BMmin}.
\smallskip

\noindent
(3) The desingularization morphism of Theorem \ref{thm:mainpairs} is functorial
in the category of pairs $(X,D)$ with a fixed ordering on the components of $X$, and with respect to \'{e}tale (or
smooth) morphisms that preserve the number of irreducible components of $X$ and $D$ passing through each point. If 
$D=0$, then the sequence of blowings-up is independent of 
the ordering of the components of $X$.
Note that desingularization preserving only snc or stable-snc singularities 
cannot be functorial with respect to \'{e}tale morphisms in general
(as in the case of functorial resolution of singularities), because a normal
crossings point becomes snc after an \'{e}tale morphism. 
\end{remarks}

The following example shows that Theorem \ref{thm:mainpairs} does not hold for more general snc singularities.

\begin{example}\label{ex:counterex}
Consider $X:=(z=x=0)\cup(z=y=0)\cup(z+xw=x+y=0)\subset\mathbb{A}_{w,x,y,z}^{4}$. Then $X$ is snc at every point 
except the origin $(w=x=y=z=0)$, so the only blowing-up permissible as the first in the sequence \eqref{eq:bluppairs}
in Theorem \ref{thm:mainpairs} has centre the origin. In the ``$w$-chart'' with coordinates $(w,x/w,y/w,z/w)$, the strict transform $X'$ of $X$ is given by the same equations as $X$,
and the exceptional divisor $D' = (w=0)$. Therefore, $(X',D')$ is snc except at $0$, and the non-snc singularity at $0$
cannot be eliminated by continuing to blow up only non-snc points.
\end{example}

Theorem \ref{thm:mainpairs} follows from a stronger version, Theorem \ref{thm:maintriples}
below, for which it
will be convenient to work with triples $(X,D,E)$ that distinguish the birational transforms of $D$ from the exceptional divisors. In this notation, $(X,D)$ has the same meaning as in 
Definition \ref{def:stablesncpairs}, and $E$ is an ordered snc divisor on X in the sense
of Definition \ref{def:orderedsnc} following (usually with all coefficients $a_k = 1$).

\begin{definition}\label{def:orderedsnc}
Let $Z$ denote a smooth variety. An \emph{(ordered) snc divisor} $E$ on $Z$ is a finite linear combination
$\sum a_k H_k$ of (ordered) subvarieties $H_k$, where each $a \in Z$ admits a coordinate neighbourhood
in which every $H_k$ is a coordinate hypersurface. We identify the \emph{support} of $E$, $\supp\, E := \sum H_k$,
with the (ordered) set of smooth hypersurfaces $\{H_k\}$. The $H_k$ are called the \emph{components} of $E$.  

Let $X$ denote a variety. An \emph{(ordered) snc divisor} 
$E$ on $X$ is a finite linear combination
$\sum a_k H_k$ of (ordered) subvarieties $H_k$, such that each $a\in X$ admits a neighbourhood $U$ and
an embedding $X|_U \hookrightarrow Z$, $Z$ smooth, where $E|_U$ is induced by an 
(ordered) snc divisor $E_Z$ on $Z$
(and each nonempty $H_k|_U$ is the restriction of a component of $E_Z$). 
Note that the \emph{components} $H_k$ of $E$ need not be irreducible (or reduced).
When all $a_k=1$, we 
again identify $E$ with the (ordered) set of smooth hypersurfaces $\{H_k\}$.
 
We also assume
that $E$ is a Weil divisor (as in Definition \ref{def:stablesncpairs}). 
\end{definition}

\begin{remark}\label{rem:orderedsnc}
The latter assumption
only excludes the possibility that a component of $E$ contain an irreducible component of $X$.
This possibility does not arise, in any case, for the exceptional divisor of a sequence
of blowings-up as given by our main theorems. If we were to allow it, the proofs of Theorems
\ref{thm:mainpairs} and \ref{thm:maintriples} would simply require an additional step to
separate and blow up such irreducible components of $X$ (which contain no stable snc
points of $(X,E)$).
\end{remark}

Let $X$ denote a variety and let $E$ denote an snc divisor on $X$. Consider a sequence of 
blowings-up
\begin{equation}\label{eq:blupseq}
X = X_0 \stackrel{\s_1}{\longleftarrow} X_1 \longleftarrow \cdots
\stackrel{\s_{t}}{\longleftarrow} X_{t}\,,
\end{equation}
where each $\s_{j+1}$ has smooth centre $C_j \subset X_j$.
Write $E_0 = E$ and, for each $j=0,1,\ldots$, set $E_{j+1} :=$ the birational transform of $E_j$ (with the
induced ordering) plus the exceptional divisor $\s_{j+1}^{-1}(C_j)$ of $\s_{j+1}$ (as the last element).

\begin{definition}\label{def1:admiss}
A smooth blowing-up $\s: X' \to X$ (i.e., a blowing-up with smooth centre $C\subset X$) is 
\emph{admissible} (or \emph{admissible} for $(X,E)$) if $C$ is \emph{snc with respect to} $E$ (where the latter means that, for each $a \in C$, there is a neighbourhood $U$ of $a$ in $X$
and an embedding $X|_U \hookrightarrow Z$ as above, where $Z$ has a coordinate
system in which $C$ is a coordinate subspace
and each component of $E_Z$ is a coordinate hyperplane).
The sequence of blowings-up \eqref{eq:blupseq} is \emph{admissible}
if each $\s_{j+1}$ is admissible for $(X_j,E_j)$. We will speak of $j$ as a ``year'' in the ``history'' of blowings-up
\eqref{eq:blupseq}.
\end{definition}

It follows from Definition \ref{def1:admiss} that, if $E_j$ is snc and $\s_{j+1}$ is admissible, then $E_{j+1}$ is snc.

\begin{definition}\label{def:stablesnctriples}
We say that $(X,D,E)$ has (or is) \emph{stable simple normal crossings} (\emph{stable-snc}) at a point $a \in X$ if $(X,D+E)$ is stable-snc at $a$. We say that $(X,D,E)$ is \emph{stable-snc} if it is stable-snc at every point.
\end{definition}

\begin{definition}\label{not:seqembedded} \emph{Transform} of a triple $(X,D,E)$. 
Consider a sequence of blowings-up
\eqref{eq:blupseq} that is admissible for $(X,E)$.
Write $D_0=D$ and $E_0 = E$. For each $j = 0,1,\ldots$\,, set
$D_{j+1}:=$ the birational transform of $D_j$, and define 
$E_{j+1}$ as above. 
\end{definition}

Comparing this notation with that of Definition \ref{not:seq}, note that, if $E = 0$, then
$\tD_j = D_j + E_j$, for each $j$. The notation of Definitions \ref{not:seq} and \ref{not:seqembedded} 
will be used throughout 
the article. Superscripts will be used to denote irreducible components of varieties.

\begin{theorem}\label{thm:maintriples}
Let $X$ denote a (reduced) variety in characteristic zero.
Let $D$ denote a $\mathbb{Q}$-Weil divisor and
$E$ an ordered simple normal crossings divisor on $X$. Then there is a sequence of admissible smooth
blowings-up \eqref{eq:blupseq} such that
\begin{enumerate}
\item $(X_t, D_t, E_t)$ has only stable-snc singularities;
\item each $\s_{j+1}$ is an isomorphism over the locus of stable-snc points of $(X_j, D_j, E_j)$.
\end{enumerate}
Moreover, the association of the desingularization sequence \eqref{eq:blupseq} to
$(X, D, E)$ is functorial in the category of triples $(X, D, E)$ with a fixed ordering on the components of
$X$, and with respect to
\'etale (or smooth) morphisms that preserve the number of irreducible components of $X$
at every point. (In the category of such triples with $D=0$, an ordering of the components of $X$ is
not necessary for functoriality.)
\end{theorem}

Theorem \ref{thm:mainpairs} is a consequence of Theorem \ref{thm:maintriples}. 

To prove Theorem \ref{thm:maintriples}, we construct the sequence of blowings-up in two main parts. We first make the
transform of $(X,E)$ stable-snc, and then perform further blowings-up to make the transform of $(X,D,E)$ stable-snc. 
Comparing this article with previous papers, the first part is rather analogous to \cite{BDV}, while the second is closer to
\cite{BV}. Nevertheless, the
main new arguments here are for the first part; the second part differs from \cite{BV} in a more technical way.
\smallskip

Several basic notions concerning the desingularization algorithm and the desingularization
invariant are prerequisite to the proofs of the main results. See the \emph{Crash course on the desingularization invariant} in \cite[Appendix]{BMmin}. We will use the ideas of the latter with
references but without recalling them in detail. (See also the summary in \cite[Sect.\,2]{BDV}.
The desingularization algorithm is used in \cite{BDV}, \cite{BMmin}, \cite{BV} mainly in the
case of a hypersurface or (weak desingularization of) an ideal. For desingularization of more
general varieties as treated here, the notion of \emph{presentation} of an invariant
(of origin in \cite{BMinv}) is a useful tool that will be recalled in Section 2 below, with examples
needed for the paper. We use the idea of \cite{BMDesAlg}, \cite{BMfunct} that, given a local invariant that admits a presentation, one can functorially construct a sequence of blowings-up 
along which the invariant never increases and eventually decreases.

Beyond Theorem \ref{thm:maintriples}, a number of techniques in this paper may be of interest  in other applications; in particular; other partial desingularization problems. In Section \ref{sec:SimDes}, for example, we give an algorithm for simultaneous desingularization of a finite collection of closed subvarieties of a given variety. 

\section{Presentation of an invariant}\label{sec:inv}
We will consider several local invariants of algebraic varieties with values in partially ordered
sets. These invariants provide different measures of singularity, and the desingularization algorithm
for an associated marked ideal (a \emph{presentation} of the invariant) is used to reduce the invariant to its
value at a general point. 

In \S\S\ref{sec:SimDes} and \ref{sec:DesProd}, we will illustrate these ideas by constructing presentations 
for two local invariants that intervene in our proofs of 
Theorems \ref{thm:mainpairs} and \ref{thm:maintriples}. The first is used to prove that any algebraic variety can be transformed to a variety
all of whose irreducible components are smooth, by a sequence of blowings-up that preserve points where
all components are already smooth (Theorem \ref{thm:SimDes}). This will be the first step in the proof of our
main result Theorem \ref{thm:mainpairs}; the approach is different from that of 
\cite{BDV} and \cite{BV}, so that
Theorem \ref{thm:mainpairs} involves an algorithm that differs from those of the latter, even in the special case 
that $X$ is a hypersurface. In the following sections, we will
remark certain simplifications of the remaining steps, relative to \cite{BDV} and \cite{BV}, that result from 
the use of Theorem \ref{thm:SimDes}.

Let $\La$ denote a partially ordered set, and let $\io$ denote a \emph{local invariant} 
with values in $\La$. This means that, given an algebraic variety $X$, there is a function 
$\io = \io_X: X \to \La$ such that,
for all $a \in X$, $\io(a)$ is an invariant of the local \'etale isomorphism class of $X$ at $a$.

We will assume that $\io$ satisfies the following three properties:
\begin{enumerate}
\item $\io$ is upper-semicontinuous; in particular, for all $a \in X$, $(\io(x) \geq \io(a)) := \{x \in X:
\io(x) \geq \io(a)\}$ is closed;

\item $\io$ is \emph{infinitesimally upper-semicontinuous}; i.e., for any smooth blowing-up
$\s: X' \to X$ such that $\io$ is locally constant on the centre of $\s$, $\io(a') \leq \io(a)$
whenever $a' \in X'$ and $a = \s(a')$;

\item any non-increasing sequence in the value set of $\io$ stabilizes.
\end{enumerate}

Properties (1) and (2) are needed for the notion of a presentation of $\io$. Property (3)
is needed to guarantee the termination of a desingularization algorithm based on the
invariant $\io$.

An important example
of a local invariant that satisfies the properties above is the \emph{Hilbert-Samuel function}
$\io(a) = H_{X,a}$ of the local ring $\cO_{X,a}$ (see Section \ref{sec:HS} and
\cite[\S1.3]{BMfunct}). The Hilbert-Samuel
function $H_{X,a} \in \IN^{\IN}$. The latter is partially ordered as follows: if $H_1, H_2 \in \IN^{\IN}$,
then $H_1 \leq H_2$ if $H_1(k) \leq H_2(k)$, for all $k \in \IN$.

The desingularization invariant is calculated using marked ideals \cite[Def.\,A.5]{BMmin} --- collections
of data that are computed iteratively on \emph{maximal contact} subspaces of increasing
codimension \cite[Def.\,A.11]{BMmin}. A \emph{marked ideal} $\ucI$ is a quintuple
$(Z,N,E,\cI,d)$, where $Z \supset N$ are smooth varieties, $E = \sum_{i=1}^s H_i$ is an snc
divisor on $Z$ that is transverse to $N$, $\cI \subset \cO_N$ is an ideal, and $d \in \IN$.
We will sometimes call $N$ a ``maximal contact subspace" by abuse of language, since it
typically arises in this way.
See \cite[{\S}A.4]{BMmin} for the important notions of \emph{admissible} blowing-up of a
marked ideal and of
 of \emph{equivalence} of marked ideals with
a common ambient variety $Z$. 

\begin{definition}\label{def:admiss}
Given a local invariant $\io$ and a variety $X$, we say that a 
sequence of blowings-up \eqref{eq:blupseq} of $X$ is \emph{admissible} for $(X,\io)$ 
or for $\io$ (or $\io$-admissible) if \eqref{eq:blupseq} is admissible for $(X,0)$ in the
sense of Definition \ref{def1:admiss} with $E=0$, and $\io$ is locally constant on each centre
of blowing up $C_j$.
\end{definition}

\begin{definition}\label{def:present}
Let $\io$ denote a local invariant satisfying properties (1) and (2) above. A \emph{presentation}
of $\io$ at $a \in X$ is a marked ideal $\ucI = (Z,N,0,\cI,d)$, where $X|_U \hookrightarrow Z$ is
a local embedding at $a$ (i.e., defined on a neighbourhood $U$ of $a$ in $X$) such that
\begin{enumerate} 
\item $(X|_U, \io)$ and $\ucI$ have the same sequences of admissible blowings-up
(i.e., a sequence of blowings-up is admissible for one if it is admissible for the other);
\item the equivalence class of $\ucI$ (over a sufficiently small neighbourhood $U$ of $a$)
depends uniquely on $\io$ and on $(Z,X|_U)$.
\end{enumerate}
\end{definition}

For example, the Hilbert-Samuel function $H_{X,\cdot}$ admits a presentation at any point.
In fact, given any local embedding $X|_U \hookrightarrow Z$ at a point $a$, $H_{X,\cdot}$
has a presentation $(Z,N,\ldots)$ at $a$ where $N$ is a minimal embedding variety of $X$
at $a$ (see \cite[Ch.\,III]{BMinv}). 

In general, even a simple local invariant need not admit a presentation
at a point of an arbitrary algebraic variety $X$. For example, does the \emph{local
embedding dimension} $e_{X,a}$ admit a presentation?

The purpose of a presentation is that, according to Definition \ref{def:present}, we can
decrease the invariant $\io$ over a given point $a$ by resolution of singularities of a 
corresponding presentation. When $\io$ decreases, we chose a new presentation and
repeat the process. Of course, when $\io$ decreases, we have not only the transform of $X$
but also an exceptional divisor; in general, therefore, we have to consider a variety together
with a simple normal crossings divisor (``boundary'') $E$.

We can generalize Definition \ref{def:admiss} to the case that $X$ is equipped with an snc
divisor $E$. In this situation, consider a sequence of $(X,E)$-admissible blowings-up 
\eqref{eq:blupseq}.
Let us write $\cE_j$ for successive birational
transforms of $E$, so that each $E_j = \cE_j$ plus the exceptional divisor of the morphism
given by the first $j$ blowings-up.
If $a \in X$, let $s(a)$ denote the number of components of $E$ at $a$. Likewise,
if $a \in X_j$, for any $j$, let $s(a)$ denote the number of components at $a$ of $\cE_j$. 
We consider
the invariant $(\io, s)$, where such pairs are ordered lexicographically, defined over an
$\io$-admissible sequence
of blowings-up \eqref{eq:blupseq}.

Suppose that $\ucI = (Z,N,0,\cI,d)$ is a presentation of $\io$ at $a \in X$, and that (near
$a$), $E$ is 
induced by an snc divisor on $Z$ (which, for simplicity, we also denote $E$).
For each component $H$ of $E$, let $\cI_H$ denote the ideal
of $H$ in $\cO_Z$ and consider the marked ideal $(\cI_H|_N, 1) := (Z,N,0,\cI_H|_N, 1)$. We
introduce the \emph{boundary} marked ideal $\cB := \sum_{H \ni a} (\cI_H|_N, 1)$ (see \cite[Def.\,A.8 and {\S}A.9]{BMinv}),
and define $\ucI^1 := \ucI + \ucB$. The equivalence class of the marked ideal $\ucI^1$ depends only
on that of $\ucI$ and on $E$, so that $\ucI^1$ is a presentation of $(\io, s)$ at
$a$ in the sense of an obvious generalization of Definition \ref{def:present}.

We define a \emph{desingularization invariant} $\inv = \inv_\io$ extending the
invariant $\io$ by 
$$
\inv(a)  = (\io(a), s(a), \inv_{\ucI^1}(a)),
$$
where $\inv_{\ucI^1}$ is the desingularization invariant $\inv_{\ucI^1}$ for the marked ideal $\ucI^1$ 
(see \cite[{App.}\,A]{BMmin}
and \cite{BMfunct}. The desingularization invariant $\inv$ is defined recursively over a sequence \eqref{eq:blupseq} of $\inv$-admissible blowings-up; i.e., for each $j$, if $\inv$ is defined over
$X=X_0 \leftarrow \cdots \leftarrow X_j$ and $\s_{j+1}$ is $\inv$-admissible, then $\inv$ extends to $X_{j+1}$,
and the properties (1)-(3) analogous to those of $\io$ above are satisfied by $\inv$ in the appropriate sense. The maximum locus of $\inv$ provides a global smooth centre of blowing up.

The desingularization invariant $\inv_{\ucJ}$ of a marked ideal $\ucJ$ depends only on the equivalence class
of $\ucJ$ and the dimension of the maximal contact subspace $N$. 
In order to get a well-defined semicontinuous
invariant $\inv_\io$, it is necessary to choose $N$ in a way that $\dim N$ has a canonical value; 
e.g., in a way that $\dim N$ depends only on $\io$ at $a$, or $\dim N$ is locally constant 
on $\{x:\,
\io(x) = \io(a)\}$. This is an important issue in \S\S\ref{sec:SimDes},\,\ref{sec:DesProd}
below.

Some of the technology of the desingularization invariant will be used in 
Sections \ref{sec:char} and \ref{sec:clean}.
Consider a sequence \eqref{eq:blupseq} of $\inv_\io$-admissible blowings-up. Let $a \in X_j$.
The desingularization invariant $\inv = \inv_\io = (\io(a), s(a), \inv_{\ucI^1}(a))$ is a sequence
$(\nu_{1}(a), s_1(a), \ldots, \nu_q(a), s_q(a), \nu_{q+1}(a))$, where $\nu_{1}(a) = \io(a)$, $s_1(a) = s(a)$ and
$\inv_{\ucI^1}(a) = (\nu_{2}(a), s_2(a), \ldots, \nu_{q+1}(a))$; each $s_k(a)$ is a nonnegative integer counting
the number of elements of a certain block of $E_j$ at $a$, $\nu_{k+1}(a)$ is a positive rational number, $1 \leq k < q$,
and $\nu_{q+1}(a)$ is either $0$ or $\infty$. The successive pairs $(\nu_{k+1}(a), s_{k+1}(a))$, $k\geq 1$, are
calculated using marked ideals $\ucI^k = (Z,N^k,\allowbreak\cE^k,\allowbreak\cI^k,d^k)$, where $N^1 = N$ and
$N^{k+1}$ had codimension $1$ in $N^k$. 

We write $\cI^k$ as the product $\cM(\ucI^k)\cdot\cR(\ucI^k)$ of its
\emph{monomial} and \emph{residual parts}. $\cM(\ucI^k)$ is the monomial part with respect
to $\cE^k$; i.e., the product of the ideals
$\cI_H$, $H \in \cE^k(a)$, each to the power $\ord_{H,a}\cI^k$,
where $\ord_{H,a}$ denotes the order along $H$ at $a$.
We set $\nu_{k+1}(a) := \ord_a \cR(\ucI^k)/d^k$ and
$\mu_{H,k+1}(a) := \ord_{H,a}\cI^k/d^k$, $H \in \cE^k(a)$; both are invariants
of the \emph{equivalence class} of $\ucI^k$ and $\dim N^k$
(see \cite[Def.\,5.10]{BMmin}). The marked ideals $\ucI^k$ are constructed iteratively (on the maximal
contact subspaces $N^k$ of decreasing dimension; the construction terminates when
$\nu_{k+1}(a) = 0$ or $\infty$. 

The passage from $\ucI^k$ to $\ucI^{k+1}$ actually involves two steps: first, from $\ucI^k$ to a 
\emph{companion ideal} $\ucG(\ucI^k)$ defined using the product decomposition of $\cI^k$ above, and
secondly from $\ucG(\ucI^k)$ to $\ucI^{k+1}$ as the \emph{coefficient ideal plus boundary}. 
For more details, see \cite[Appendix]{BMmin} and \cite[Sect.\,2]{BDV}.

\subsection{Simultaneous desingularization of a collection of varieties}\label{sec:SimDes}

\begin{theorem}\label{thm:SimDes}
Let $X$ denote a (reduced) algebraic variety $X$. Then there is a finite sequence of admissible smooth blowings-up \eqref{eq:bluppairs} such that
\begin{enumerate}
\item every irreducible component of $X_t$ is smooth;
\item each $\s_{j+1}$ is an isomorphism over the locus of points where all components of $X_j$ are smooth.
\end{enumerate}
Moreover, given an snc divisor $E$ on $X$, there is a sequence of smooth blowings-up as above  
which is admissible for $(X,E)$. The association of the desingularization sequence to
$(X,E)$ is functorial with respect to
\'{e}tale morphisms that preserve the number of irreducible components of $X$
at every point.
\end{theorem}

\begin{remark}\label{rem:order} Consider two local invariants $\io_1,\,\io_2$ with values
in partially-ordered sets $\La_1,\,\La_2$, respectively. Given a variety $X$, we have
$(\io_1,\io_2): X \to \La_1 \times \La_2$. There are two natural partial orders on $\La_1 \times \La_2$:
(1) the \emph{product order}, $(\la_1,\la_2) \geq (\ka_1,\ka_2)$ if $\la_1 \geq \ka_1$ and
$\la_2 \geq \ka_2$, and (2) the lexicographic order. Clearly, for either order, $(\io_1,\io_2)$ is semicontinuous, and infinitesimally semicontinuous, and any non-increasing sequence in its
value set stabilizes. The maximal loci of $(\io_1,\io_2)$ with respect
to the two orders coincide locally at a point of $X$, but not necessarily globally.

Suppose that
$\ucI_1,\,\ucI_2$ are presentations of $\io_1,\,\io_2$ (respectively) at a point $a \in X$. Assume that 
$\ucI_1,\,\ucI_2$ have common ambient variety $Z$ and common maximal contact subvariety $N$.
Then $\ucI_1 + \ucI_2$ is a presentation of $(\io_1,\io_2)$, with respect to either order, but
the desingularization algorithms based on $(\io_1,\io_2)$, for the two orders need not coincide:
the invariant tells us in what order to assemble the local centres of blowing up given by 
presentations, and this depends on the partial order on $\La_1 \times \La_2$.
\end{remark}

\begin{proof}[Proof of Theorem \ref{thm:SimDes}]
 Let $X^{(1)},\ldots,X^{(m)}$ denote the irreducible components of $X$. Consider the local invariant 
 $\io_{X,a}:=(H_{X,a},H_{X^{(1)},a},\ldots,H_{X^{(m)},a})$, $a\in X$, given by the Hilbert-Samuel functions of the local rings of $X$ and the $X^{(i)}$ at $a$ (with $H_{X^{(i)},a}=0$ if $a\notin X^{(i)}$).
 
We consider $(H,H_1,\ldots,H_m) \in (\IN^{\IN})^{m+1}$ as a pair $(H, (H_1,\ldots,H_m))
\in \IN^{\IN} \times (\IN^{\IN})^{m}$, and we use the product order on $\{(H_1,\ldots,H_m)
\in (\IN^{\IN})^{m}$, but the lexicographic order on pairs in $\IN^{\IN} \times (\IN^{\IN})^{m}$.

Given $a \in X$, there is a local embedding $X|_U \hookrightarrow Z$ such that $E$ is induced
by an snc divisor on $Z$, and the Hilbert-Samuel functions $H_{X,\cdot}$ and $H_{X^{(i)},\cdot}$,
$i=1,\ldots,m$, admit presentations $\ucI=(Z,N,0,\cI,d)$ and $\ucI^{(i)}=(Z,N,0,\cI^{(i)},d_i)$, $i=1,\ldots,m$, where $N$ is a minimal embedding variety for $X$ at $a$. Then
$$
\ucH := \ucI + \sum_{\{i:\,a\in X^{(i)}\}}\ucI^{(i)}
$$
is a presentation of $\io_{X,\cdot}$ at $a$. We can extend $\io = \io_{X,\cdot}$ to a desingularization
invariant $\inv_{\io} = (\io_{X,a}, s(a), \inv_{\ucI^1}(a))$, as above, where $\ucI^1 = \ucH + \ucB$. 
Since we are using the product order
on $(\IN^{\IN})^{m}$, $\inv_{\io}$ and the 
resulting desingularization algorithm do not depend on the ordering
of the components $X^{(i)}$.

We modify this desingularization algorithm by making a selection from the sequence 
of centres of blowings-up,
in the following way.  At each step, consider the locus of points $W$ where all components of (the
transform of) $X$ are smooth. Of course $W$ is open in $X$. Moreover, the maximum locus of
$\inv_{\io}$ in $X\backslash W$ is closed in $X$, since the Hilbert-Samuel function distinguishes
smooth from singular points (so that $\io_{X,\cdot}$ distinguishes points where all components
are smooth from points where one is singular). Therefore, at each step, we can blow up
the maximum locus of $\inv_{\io}$ in $X\backslash W$, and eventually $W = X$. 
\end{proof}

\begin{remark}\label{rem:laterSimDes}
More general families of varieties can also be simultaneously desingularized
as in Theorem \ref{thm:SimDes}. See Step 3 of the proof of Theorem 
\ref{thm:mainDEqualZero} for another application of the idea above.
\end{remark}

\subsection{Presentation of the number of irreducible components}\label{sec:DesProd}
Let $X$ denote a reduced algebraic variety. Assume that all irreducible components $X^{(i)}$ of
$X$ are smooth. For all $a\in X$ let $\kappa(a)=\kappa_{X}(a)$ denote the number of irreducible components of $X$ at $a$. 

Let $a \in X$. Consider a local embedding $X|_U \hookrightarrow Z$ at $a$,
and a smooth subvariety $N$ of $Z$ containing $\bigcap_{\{i:\, a \in X^{(i)}\}} X^{(i)}$ 
(restricted to $U$).
For each $i$,
 $\ucI^{(i)}$ denote the marked ideal $\ucI^{(i)}=(Z,N,0,\cI^{(i)}|_N,1)$, where $\cI^{(i)}$ is the ideal 
of $X^{(i)}$ in $\cO_Z$. Define $\ucI_{\Pi(X)}^N:=\sum_{\{i:\, a\in X^{(i)}\}}\ucI^{(i)}$. 
Clearly, $\cosupp \ucI_{\Pi(X)}^N$ is the constant locus $(\ka(x) = \ka(a))$ of $\ka$
if $U$ is small enough. 

Consider
a blowing-up $\s: X' \to X$ over $U$, with smooth centre in $\cosupp \ucI_{\Pi(X)}^N$.
Then the transform of each marked ideal $\ucI^{(i)}$ is given by the ideal $u^{-1}\cdot\sigma^{*}(I^{(i)})|_{N'}$,
where $u$ is (a local generator of the ideal of) the exception divisor of $\s$, and $N'$ is the
strict transform of $N$. Since $X^{(i)}$ is smooth, $u^{-1}\cdot\sigma^{*}(I^{(i)})$ defines the
strict transform of $X^{(i)}$. Therefore, the transform of the marked ideal $\ucI_{\Pi(X)}^N$
equals $\ucI_{\Pi(X')}^{N'}$, where $X'$ is the strict transform of $X$. It is then easy to see
that $\ucI_{\Pi(X)}^N$ is a presentation of the invariant $\ka$ at $a$.

We can
use the marked ideal $\ucI_{\Pi(X)}^{N}$ to extend $\ka$ to a desingularization invariant
$\inv_\ka^N(a) = (\ka(a), s(a), \inv_{\ucI^1}(a))$, as above. (In particular, $\ucI^1 = \ucI_{\Pi(X)}^{N}$ plus
boundary, where we are allowing a given snc divisor $E$).

\begin{remark}\label{rem:kinv}
Recall that $\inv_{\ucJ}$ above depends on $\dim N$ and that, in order to get a global
desingularization algorithm, we need to choose $N$ in a way that $\dim N$ has a
canonical value. We can achieve this simply by taking $N = X^{(i)}$, for
any $i$ such that $X^{(i)}$ is of minimal dimension among the components of $X$ at $a$.
With $N$ chosen in this way, the equivalence class of the marked ideal $\ucI_{\Pi(X)}^N$ 
plus boundary depends only on $X$ and $E$ at $a$, and $\inv_\ka := \inv_\ka^N$ is
globally semicontinuous.
\end{remark}

In Section \ref{sec:char}, we will use $\inv_\ka$ to give a characterization of the condition stable-snc. 


\section{Characterization of stable-snc singularities of a variety with snc divisor}\label{sec:char}
Consider an algebraic variety $X$ with snc divisor $E$. Assume that all irreducible
components of $X$ are smooth. The main purpose of this section is to characterize stable-snc singularities of $(X_j, E_j)$, over a sequence \eqref{eq:blupseq} of admissible blowings-up (see Theorem \ref{thm:char}). This section is a
generalization of \cite[Sect.\,3]{BDV}, but a presentation of the invariant $\ka = \ka_X$ (\S\ref{sec:DesProd})
plays a new role, and the assumption of smooth irreducible components allows some simplification.

Recall the following geometric characterization of stable-snc singularities of $X$, from Lemma \ref{lem:stablesnc}.
Let $a \in X$ and let $Z$ denote a smooth local embedding variety of $X$ at $a$. Let $X^{(1)},\ldots,X^{(m)}$
denote the irreducible components of $X$ at $a$ and let $c_i$ denote the codimension of $X^{(i)}$ in $Z$, for each $i$.
Then $X$ is stable-snc at $a$ if and only if (the scheme-theoretic intersection) $\cap_{i=1}^{m}X_{i}$ is smooth
and of codimension
\begin{equation}\label{eq:codimstablesnc}
c = \sum_{i=1}^{m}c_{i}-(m-1)(\dim Z_a - e_{X,a})
\end{equation}
at $a$, where $e_{X,a}$ denotes the minimal embedding dimension of $X$ at $a$.

\begin{example}
Let $X= X^{(1)} \cup X^{(2)} \subset\mathbb{A}^5$, where $X^{(1)}=(x=y=0)$ and $X^{(2)}=(x+uz=y+ut=0)$. Then 
$\mathbb{A}^5$ is a minimal embedding variety at the origin, and $X^{(1)}\cap X^{(2)}=(x=y=uz=ut=0)$. Since 
$X^{(1)}\cap X^{(2)}$  is 
not smooth at $0$, $X$ is not stable-snc at $0$. On the other hand, $X^{(1)}\cap X^{(2)}$ coincides with
$(x=y=z=t=0)$ at a nonzero point $a$ of the latter, so that $X$ is stable-snc at $a$, by \eqref{eq:codimstablesnc}.
\end{example}

The following definition describes the special values that $\inv_\ka$ can take at a stable-snc point in any
year $j$ of a history of $\inv_\ka$-admissible blowings-up (see Lemma \ref{lem:snc}).

\begin{definition}\label{def:specialinv}
Consider $c=(c_1,c_2,\ldots,c_{m})\in\mathbb{N}^m$, and $s=(s_1,\ldots,s_d)\in\mathbb{N}^d$, set
\[
\inv_{c,s}:=(m,s_1,1,s_2,\ldots,1,s_d,1,0,\ldots,1,0,\infty),
\]
where the total number of pairs (before $\infty$) is
\[
r:= |c|+|s|-\max\{c_i\}, \quad |c| := \sum_{i=1}^m c_i,\  |s| := \sum_{k=1}^d s_k
\]
(cf. \cite[Def.\,2.1]{BDV}).
\end{definition}

The $s_k$ in Definition \ref{def:specialinv} will represent the sizes of certain blocks of exceptional divisors. The $c_i$ 
will eventually be the codimensions of the 
components of $X$ in a local minimal embedding variety. The term $\max\{c_i\}$ appears in the expression for $r$ because we are using a presentation of $\ka$ with maximal contact variety $N =$ a component of $X$ of smallest dimension (see Remark \ref{rem:kinv}).

Theorem \ref{thm:char} shows, in particular, that in year zero (i.e., before any blowings-up),
stable-snc singularities can be characterized using the invariant $\inv_\ka$ together with 
the dimensions of 
a minimal embedding variety and the irreducible components of $X$. The first example
following shows that $\inv_\ka$ alone is not enough to characterize stable-snc, while the
second shows that we cannot replace $\inv_\ka$ by the desingularization invariant $\inv_X$
based on the Hilbert-Samuel function.

\begin{examples}\label{ex:invnotdeterminingstablesncX}
(1) Consider $X= (xyz=0)$ and $Y = (x=yz=0)\cup(z=0)$ in $\mathbb{A}^3$. In each case, $\ka$ has a
presentation at $0$ given by the marked ideal $(\mathbb{A}^3,(z=0),0,(x,y),1)$ (see Remark
\ref{rem:kinv}) and $\inv_\ka(0) =(3,0,1,0,1,0,\infty)$. But $X$ is stable-snc while $Y$ is snc but
not stable-snc at $0$.
\smallskip

\noindent
(2) Consider $X = (x=y=0) \cup (w=z=0)$ and $Y = (x=y=0)\cup(x+w^2=y+wz=0)$ in 
$\mathbb{A}^{4}$, with $E = 0$. The $X$ is stable-snc, while $Y$ is not because the 
scheme-theoretic intersection of its components is not smooth at $0$. $X$ and $Y$ each have
minimal embedding dimension $4$ and two components of dimension $2$ at $0$. The ideal
of $Y$ is $(x, y)\cap(x+w^2,y+wz)=(x^2+xw^2,xy+yw^2,y^2+ywz,yw-xz)$. Then $H_{X,0}
= H_{Y,0}$, and $\inv_{X}(0) = \inv_Y(0)=(H,0,1,0,1,0,1,0,\infty)$, where $H =H_{X,0}$.
\end{examples}

\begin{definition}\label{def:sig}
We consider sequences $\Omega=(e,c_1,\ldots,c_{m})$, where $e\in\mathbb{N}$ and $c_1\geq c_2\geq\ldots\geq c_m\geq0$ are integers. Let $\Sigma_\Omega = \Sigma_\Omega(X)$ denote the set of points $a \in X$ where $X$ has local embedding dimension $e_{X,a} = e$ and exactly 
$m$ irreducible components of $X$ of codimensions $c_1,\ldots,c_m$ in a minimal embedding
variety. See also Definition \ref{def:sigmapq}.
\end{definition}

The sequence of codimensions $c_i$ is taken in decreasing order in this definition
because we do not want $\Sigma_{\Omega}$ to depend on an ordering of the $c_i$ 
(since $\inv_{c,s}$ does not depend on an ordering).  

The following results deal with stable-snc singularities of the transforms $(X_j,E_j)$ of $(X,E)$
over a sequence \eqref{eq:blupseq} of $\inv_\ka$-admissible blowings-up. We are assuming
that all irreducible components of $X$ are smooth. For brevity of
notation, we will write $(X_j,E_j)$ simply as $(X,E)$. See \cite[\S A.2]{BMmin} or
\cite[Sect.\,2]{BDV} for the definition of the blocks of exceptional divisors $E^i(a)$ that
are counted by the invariants $s_i(a)$.

\begin{lemma}\label{lem:snc} \emph{(Compare with \cite[Lemma 3.1]{BDV}.)}  Suppose that
all irreducible components of $X$ are smooth.
Consider $(X,E) = (X_q,E_q)$, in some year $q$ of a history 
\eqref{eq:blupseq} of $\inv_\ka$-admissible blowings-up. Let $a \in X$. 
If $(X,E)$ is stable-snc at $a$, then 
$\inv_\ka(a) = \inv_{c,s}$, for some $s=(s_1,\ldots,s_d)$, with $c=(c_1,\ldots,c_m)$, where $m = \ka(a)$ and
each $c_k$ is the codimension of an irreducible component $X^{(k)}$ of $X$ at $a$ in a local minimal embedding
variety for $X$ (so that $a \in \Sigma_\Omega(X)$, where 
$\Omega=(e_{X,a},c_1,\ldots,c_m)$).
\end{lemma}

\begin{proof}
Suppose that $X$ has
$m$ (smooth) irreducible components $X^{(1)},\ldots,X^{(m)}$ at $a$ (so that 
$\ka_X(a) = m$), of codimensions $c_1,\ldots,c_m$,
respectively, in a local minimal embedding variety $Z$ of $X$ at $a$.
Assume (without loss of generality) that 
$c_1=\max\{c_i\}$.
Let $I^{(k)}$ denote the ideal of $X^{(k)}$ in $\cO_Z$ at $a$, $k=1,\ldots,m$. 
As in \S2.2, $\inv_\ka(a) = (\ka(a), s_1(a), \inv_{\ucI^1}(a))$,
where $\ucI^1 = \sum \ucI^{(i)}|_{N^1} +$ boundary and $N^1=X^{(1)}$. 

Let $f_{k,l}$, $l=1,\ldots,c_k$, denote generators of the ideal $I^{(k)}$ (with linearly independent gradients), $k=1,\ldots,m$.
Let $u_1^j$, $j=1,\ldots, s_1(a)$, denote generators of the ideals of the components
of $E^1(a)$. Then
\begin{equation}\label{eq:coeff}
\ucI^1 = (Z,\, N^1=X^{(1)},\, \cE^1(a),\, \cI^1 = (\{f_{k,l}:\, k\geq 2\},\, \{u_1^j\})|_{N^1},\, 1),
\end{equation}
where $\cE^1(a) = E(a)\setminus E^1(a)$ and $E(a)$ denotes the set of
components of $E$ at $a$.
The argument is now very similar to the proof of \cite[Lemma 3.1]{BDV}.

We factor $\ucI^1$ as the product $\cM(\ucI^1)\cdot\cR(\ucI^1)$ of its monomial
and residual parts; in particular, $\cM(\ucI^1)$ is generated by a monomial $m_1$
in the components of $\cE^1(a)$.

Since $(X,E)$ is stable-snc at $a$,
the generators of $\cI^1$ in \eqref{eq:coeff} are part
of a regular coordinate system. It follows that $\cM(\ucI^1) = 1$ (since none of
these generators define elements of $\cE^1(a)$); i.e., all
$\mu_{H,2}(a) = 0$. Since $\ucI^1$ has maximal order, $(\inv_\ka)_{3/2}(a) = (m,s_1,1)$, and
the companion ideal $\ucJ^1 = \ucI^1$.

We can continue the computation of $\inv_\ka$,
choosing the $f_{k,l}$ and the $u_{i}^{j}$ successively as hypersurfaces of maximal contact
to pass to the coefficient ideal plus boundary
$\underline{\mathcal{I}}^p$, $p = 2,\ldots$\,.
At each step, $\cM(\ucI^p) = 1$ (in particular, $\mu_{H,p}(a)=0$ for every $H$),
and $\underline{\mathcal{I}}^p$ is of maximal order, $=1$.
Therefore, $\nu_{p +1}=1$ and
$\underline{\mathcal{I}}^p$ equals the following companion ideal
$\underline{\mathcal{J}}^p$.
Once all $f_{k,l}$ and $u_{i}^{j}$ have been used as hypersurfaces of maximal contact, we
get coefficient ideal $= 0$. Therefore, $\inv_\ka(a)$ has last entry $=\infty$ and $r$ pairs
before $\infty$.
\end{proof}

\begin{lemma}\label{lem:norm} \emph{(Compare with \cite[Lemma 3.3]{BDV}.)}
Again consider $(X,E) = (X_q,E_q)$, in some year $q$ of a history 
\eqref{eq:blupseq} of $\inv_\ka$-admissible blowings-up, and let $a \in X$. Assume that $X$ has
$m$ irreducible components $X^{(1)},\ldots,X^{(m)}$ at $a$ (all smooth).
Let $f_h,\, h=1,\ldots,p$, denote generators of the ideal of $\cap_{k=1}^{m} X^{(k)}$ in $\cO_N$ at $a$, where $N$ is a component $X^{(k)}$ of smallest dimension
(say $N = X^{(1)}$, without loss of generality). 
Let $u_{i}^{j},\, j = 1,\ldots, s_i$, denote generators of the ideals of the elements of
$E^i(a)|_N,\, i=1,\ldots, d$, and write $s = (s_1,\ldots,s_d)$.

Assume that $\inv_\ka(a)=\inv_{c,s}$, with $c=(c_1,\ldots,c_m)$. Set $r:=|c|+|s|-\max\{c_i\}$. 
Then there is an injection $\{1,\ldots,r\} \to \{f_h, u_{i}^{j}\}$, which
we denote  $l\mapsto g_l$, and a regular system of coordinates $(x_1,\ldots,x_n)$ for $N$ at
$a$ ($n\geq r$), such that
\begin{equation}\label{eq:localform}
g_l=\xi_l+x_l \cdot\prod_{i=1}^{l-1}m_{i},\quad l=1,\ldots,r,
\end{equation}
where each $\xi_l$ is in the ideal generated by $(x_1,\ldots, x_{l-1})$ and each $m_{i}$ is a monomial in generators of the ideals of the elements $H$ of
$\mathcal{E}^i(a)= E(a)\setminus E^1(a)\cup ...\cup E^i(a)$, each raised to the power
$\mu_{H,i+1}(a)$.
\end{lemma}

\begin{remark}\label{rem:norm}
Suppose that the irreducible components $X^{(k)}$ of $X$ have codimensions $c_k$ in a minimal
embedding variety for $X$ at $a$. Then we can take $\{f_h\} := \{f_{k,j}|_N\}_{k\geq 2}$, where the 
$f_{k,j}$, $l=1,\ldots,c_k$ 
denote local generators of the ideal $I^{(k)}$ of $X^{(k)}$ in a minimal embedding variety for $X$ at $a$.
In this case, the mapping $l\mapsto g_l$ of the lemma is bijective.
\end{remark}

\begin{proof}[Proof of Lemma \ref{lem:norm}]
As in the proof of Lemma \ref{lem:snc}, $\inv_\ka(a) = (\ka(a), s_1(a), \inv_{\ucI^1}(a))$,
where $\ucI^1 = \sum \ucI^{(i)}|_{N^1} +$ boundary, $N^1 = N$ and 
$\ucI^1 = (Z,\, N^1,\, \cE^1(a),\, \cI^1 = (\{f_h\},\, \{u_1^j\}),\, 1)$ (with $Z$ a local embedding variety).
If $(\inv_\ka)_{3/2}(a) = (m,s_1,1)$, then there exists $g_1\in
\{f_h\} \cup \{u_1^j\}$ such that $x_1 := (m_1)^{-1}\cdot g_1|_{N^1}\in \mathcal{R}(\underline{\mathcal{I}}^1)$ has order 1 
at $a$, and the 
companion ideal $\ucJ^1 = (Z,N^1,\cE^1(a),\cR(\ucI^1),1)$. We can take $N^2 :=
(x_1=0) \subset N^1$ as the next maximal contact subspace. Then the coefficient ideal plus boundary is
\begin{equation*}
\ucI^2 = \left(Z, N^2, \cE^2(a)=\cE^1(a)\setminus E^2(a),
\left(\cR(\ucI^1) + (u_2^1,\ldots,u_2^{s_2})\right)|_{N^2}, 1\right).
\end{equation*}
We can again repeat the argument, as in \cite[Sect.\,3]{BDV}, and the process ends after $r$ steps. 
\end{proof}

\begin{remark}\label{rem:mult1}
In the proof above, we see that, if the truncated invariant
$(\inv_\ka)_{k+1/2}(a)\allowbreak =(\inv_{c,s})_{k+1/2}$, where $0\leq k<r=|c|+|s|-c_1$, then, for every
$p \leq k+1$, the coefficient ideal plus boundary $\underline{\mathcal{I}}^p$ (or an equivalent marked ideal) has associated multiplicity $=1$. Comparing with \cite[Remark 3.6]{BDV}, note that a condition analogous to ``$a \in \Sigma_p$'' in
the latter is not needed here because we are assuming all irreducible components of $X$ at $a$ are smooth.
\end{remark}

\begin{theorem}[Characterization of stable-snc]\label{thm:char}
Consider $(X,E) = (X_q,E_q)$, in some year $q$ of a history 
\eqref{eq:blupseq} of $\inv_\ka$-admissible blowings-up.
Let $a \in X$, and let $e = e_{X,a}$. Assume that the irreducible components, $X^{(k)}$,
$k = 1,\ldots, m = \ka(a)$ of $X$ at $a$ are smooth and of dimensions $e-c_k$, respectively. Then $(X,E)$ is stable-snc at $a$ if and only if
\begin{enumerate}
\item $a\in\Sigma_{\Omega}(X)$, where  $\Omega=(e,c_1,\ldots,c_m)$;
\item $\pinv(a)=\inv_{c,s}$, for some $s=(s_1,\ldots,s_d)$;
\item $\mu_{H,i+1}(a) = 0$, for all $i\geq 1$ and all $H \in \cE^i(a)$.
\end{enumerate}
\end{theorem}

\begin{proof}
``Only if'' is immediate from Lemma \ref{lem:snc}. On the other hand, assume conditions (1), (2) and (3). By (3),
\eqref{eq:localform} holds with all $m_i=1$. Then, by Lemma \ref{lem:norm}, the scheme-theoretic intersection of the components of $X$ and $E$ at $a$ is smooth, and \eqref{eq:codimstablesnc} holds. So 
$(X,E)$ is stable-snc.
\end{proof}


\section{Cleaning}\label{sec:clean} We recall the cleaning technique introduced in \cite{BMmin} and developed
in \cite[Section 4]{BDV} under conditions that also apply here (in fact, in a more straightforward way).

Assume that all irreducible components of $X$ are smooth.
According to Theorem \ref{thm:char}, if $a\in\Sigma_{\Omega}(X)$ and $\inv_\ka(a)=\inv_{c,s}$,
then $(X,E)$ is stable-snc at $a$ if and only if the invariants $\mu_{H,k+1}(a)=0$, for every
$k\geq1$. In this section we study the \emph{cleaning} blowings-up used to get the latter condition.

Cleaning blowings-up are not necessarily $\inv_\ka$-admissible. In the general cleaning algorithm of \cite[Sect.\,2]{BMmin}, therefore, the invariant $\inv = \inv_X$ that is used is not defined
in a natural way over a cleaning sequence, so that, after cleaning, we assume
we are in year zero for the definition of the invariant. Over the particular cleaning
sequences needed here, however, we can define a modified $\inv_\ka$ which
remains upper semicontinuous and infinitesimally upper semicontinuous, and show that maximal
contact subspaces exist in every codimension involved; this is a
consequence of Lemma \ref{lem:norm} and Remark \ref{rem:mult1} (see Remarks
\ref{rem:clean}).

Consider a point $a$ in the locus $S := ((\inv_\ka)_k = (\inv_{c,s})_k)$ for the 
truncated invariant, where  $k\geq 1$
(in some year $q$ of a history 
\eqref{eq:blupseq} of $\inv_\ka$-admissible blowings-up). In some neighbourhood of $a$,
$S$ is the cosupport of a marked ideal (a coefficient ideal
plus boundary) $\underline{\mathcal{I}}^k = (\mathcal{I}^k,d^k)
=(Z,N^k,\mathcal{E}^k(a),\mathcal{I}^k,d^k)$, where $N^k$ is a maximal contact subspace
of codimension $k-1$ in $N^1$ and $d^k=1$ (see Remark \ref{rem:mult1}). Recall that
$\mathcal{E}^k(a)=E(a)\setminus E^1(a)\cup ...\cup E^{k}(a)$, where
the block $E^{k}(a)$ defines the boundary.

The ideal $\mathcal{I}^k = \mathcal{M}(\underline{\mathcal{I}}^k )\cdot\mathcal{R}(\underline{\mathcal{I}}^k )$ (the product of its monomial and residual parts). The
monomial part $\mathcal{M}(\underline{\mathcal{I}}^k )$ is the product of the ideals
$\mathcal{I}_H|_{N^k}$ (where $H \in \cE^k(a)$), each to the power $\mu_{H,k+1}(a)$
(since $d^k=1$).

Let $\underline{\mathcal{M}}(\underline{\mathcal{I}}^k)$ denote the monomial
marked ideal $(\mathcal{M}(\underline{\mathcal{I}}^k), d^k)
= (\mathcal{M}(\underline{\mathcal{I}}^k), 1)$. Then $\cosupp \underline{\mathcal{M}}(\underline{\mathcal{I}}^k) \subset \cosupp \ucI^k$ and any admissible sequence of
blowings up of $\underline{\mathcal{M}}(\underline{\mathcal{I}}^k)$ is admissible
for $\ucI^k$. 

\begin{definition}\label{def:clean}
\emph{Cleaning} of the locus $S = ((\inv_\ka)_k = (\inv_{c,s})_k)$ means the sequence
of blowings-up obtained from desingularization of the monomial marked ideal
$\underline{\mathcal{M}}(\underline{\mathcal{I}}^k)$ (in a neighbourhood of
any point of $S$) \cite[Sect.\,5, Step II, Case A]
{BMfunct}, \cite[Sect.\,2]{BMmin}.
\end{definition}

The centres of the cleaning blowings-up are invariantly
defined closed subspaces of $((\inv_\ka)_k \geq (\inv_{c,s})_k)$. Definition \ref{def:clean} is simpler
than the analogous definition \cite[Def.\,4.2]{BDV} because of our assumption that all components of
$X$ are smooth.

\begin{remarks}\label{rem:clean}
The blowings-up $\s$ involved in  desingularization of $\underline{\mathcal{M}}(\underline{\mathcal{I}}^k)$ are 
$(\inv_\ka)_k$-admissible: Let $C$ denote the centre of $\s$.
Then $C$ is snc with respect to $E$
because, in the notation above, $C$ lies inside every element of $E^1(a) \cup \cdots E^k(a)$ and $C$ is snc with respect to $\cE^k(a)$. Since $C \subset S$, it follows
that $\s$ is $(\inv_\ka)_k$-admissible. By Lemma \ref{lem:snc}, $C$
contains no stable-snc points (since some $\mu_{H,k+1}(a) \neq 0$, for all $a \in C$).

Since $d_k=1$, $C$ is of the form $N^k \cap H$, for a single $H \in \cE^k(a)$;
i.e., $C$ is of codimension $1$ in $N^k$. Therefore, $\s$ induces an isomorphism
$(N^k)' \to N^k$, where $(N^k)'$ denotes the strict transform of $N^k$.
\end{remarks}

\begin{lemma}\label{lem:clean}
Assume that $\inv_\ka \leq \inv_{c,s}$ on $X = X_q$, in some year $q$ of a history 
\eqref{eq:blupseq} of $\inv_\ka$-admissible blowings-up. Consider the cleaning sequence for
$(\pinv_k = (\inv_{c,s})_k)$ (Definition \ref{def:clean}). Then, over the cleaning
sequence, we can define maximal contact subspaces of every codimension involved,
as well as (a modification of) $\inv_\ka$ which remains both semicontinuous and infinitesimally semicontinuous.
\end{lemma}

The proof is the same as that of \cite[Lemma 3.20]{BDV} (changing $\inv$ to $\inv_\ka$).

\begin{remark}\label{rem:why}
After cleaning the loci $((\inv_\ka)_k = (\inv_{c,s})_k)$, for all $k$, we will apply further blowings-up to make 
$(X,E)$ stable-snc on $(\inv_\ka = \inv_{c,s})$ (see Section \ref{sec:algDEqualsZero}, Step 3). We will then 
continue to blow up with closed centres which lie in
the complement of the stable-snc locus
$\{\text{stable-snc}\}$ (Section \ref{sec:algDEqualsZero}). The purpose of defining
$\inv_\ka$ over the cleaning sequences is to ensure that, in the complement
of $\{\text{stable-snc}\}$, we will only have to consider values $\inv_{c',s'} < \inv_{c,s}$
in order to resolve all but $\{\text{stable-snc}\}$ after finitely many steps. If,
after cleaning $(\inv_\ka = \inv_{c,s})$, we were to apply the resolution algorithm in the
complement of $\{\text{stable-snc}\}$, beginning as if in year zero, we might introduce
points where $\inv_\ka = \inv_{c',s'} > \inv_{c,s}$.
\end{remark}


\section{Desingularization of a variety preserving stable-snc singularities}
\label{sec:algDEqualsZero}

The purpose of this section is to give an algorithm for our main theorem in the case that $D=0$.
We prove the following result.

\begin{theorem}\label{thm:mainDEqualZero}
Let $X$ denote a reduced algebraic variety and let $E$ be an snc divisor on $X$. Then there is a 
sequence of admissible smooth blowings-up \eqref{eq:blupseq}, such that
\begin{enumerate}
\item $(X_t, E_t)$ has only stable-snc singularities;
\item each blowing-up $\sigma_{j+1}$ is an isomorphism over the locus of stable-snc points of
$(X_j, E_j)$.
\end{enumerate}
\end{theorem}

\begin{proof}
We will break the algorithm into three main steps, with the second and third to be iterated several
times.
\smallskip

\noindent
\textbf{Step 1.} We first reduce to the case that all irreducible components of $X$ are smooth, 
using Theorem \ref{thm:SimDes}.
\smallskip

Now let $\cS$ denote the set of all special values $\inv_{c,s}$, $c = (c_1,\ldots, c_m)$, 
$s=(s_1,\ldots, s_d)$ 
(see Definition \ref{def:specialinv}). Then $\cS$ is totally ordered (lexicographically).

Consider the desingularization sequence determined by the invariant $\inv_\ka$, defined in 
\S\ref{sec:DesProd}.
\smallskip

\noindent
\textbf{Step 2.} We follow the desingularization algorithm determined by $\inv_\ka$ (i.e., the sequence of 
blowings-up with successive centres given by the maximum locus of $\inv_\ka$) until the maximum of 
$\inv_\ka$ is a value $\tau$ in $\cS$ for the first time. We then blow up any irreducible 
component of the maximum locus $(\inv_\ka = \tau)$ that contains no stable-snc points.
The result is that
$(X,E)$ ($=(X_j,E_j)$, for some $j$) is generically stable-snc on every component of the locus $(\inv_\ka = \tau)$. (The latter
may now be empty.)  

We now clean the locus $((\inv_\ka)_k = (\tau)_k)$
of the truncated invariant, for every $k$, beginning with the largest $k$; see Section \ref{sec:clean}.
The result of cleaning is that the invariants $\mu_{H,k+1}= 0$ on $(\inv_\ka = \tau)$, for all $H\in E$ and $k\geq1$.
Recall that, for each $k$, the cleaning blowings-up are given by desingularization of a monomial marked ideal
$\underline{\mathcal{M}}(\underline{\mathcal{I}}^k)$ with cosupport in $((\inv_\ka)_k \geq (\tau)_k)$.
The cleaning blowings-up may be nontrivial even in the case that $(\inv_\ka = \tau)= \emptyset$, but are needed
even in this case to guarantee functoriality.

Cleaning involves blowing up only points where $\mu_{H,k+1}>0$, for some $k$, so never involves blowing up
stable-snc points (by Theorem \ref{thm:char}). After cleaning, we have the normal forms of Lemma \ref{lem:norm}
with all monomials $m_i=1$.
\smallskip

Recall that the characterization of stable-snc points $a$ given by Theorem \ref{thm:char} involves the
the minimal embedding dimension $e_{X,a}$. After Step 2 above, it need not be true that $e_X$ is
constant on each irreducible component of the locus $(\inv_\ka = \tau)$ (although the number 
of irreducible components of $X$ is constant). The purpose of Step 3 following is to make $e_X$ constant
on components of the maximal locus, in order to apply Theorem \ref{thm:char}.
\smallskip

\noindent
\textbf{Step 3.} If the locus $T := (\inv_\ka = \tau)$ is nonempty after Step 2, then $T$ is the maximum
locus of $\inv_\ka$, each irreducible component of $T$ is generically stable-snc, and all $\mu_{H,k+1}=0$
on $T$. We now apply the algorithm for simultaneous desingularization of the pair $(X,T)$, as in \S\ref{sec:SimDes};
i.e., the sequence of blowings-up given by the maximum loci of the invariant $\inv_\io$ determined by 
$\io := (H_X,H_T)$, with the lexicographic ordering of such pairs. 
Since $T$ is smooth, the invariant $\inv_\io$ has the form $(\io, s, \inv_{\ucI^1})$, where $\ucI^1$
is the marked ideal given by a presentation of the Hilbert-Samuel function $H_X$ restricted to $N = T$, plus a boundary.
We blow up following the algorithm until $H_X$ and therefore the embedding dimension $e_{X,a}$ is constant
on every component of $T$. The centres of all blowings-up involved lie in $T$ (thus are $\inv_\ka$-admissible) and
contain no stable-snc points; all $\mu_{H,k+1}$ remain zero on $T$.
\smallskip

After Step 3, every component of $T = (\inv_\ka = \tau)$ lies in some $\Sigma_{\Omega}$. By
Theorem \ref{thm:char}, $(X,E)$ is stable-snc at every point of $T$, and therefore in some neighbourhood of $T$.

We can now iterate Steps 2 and 3 in the complement of $T$. All centres of blowing up involved are closed in $X$
because they contain no stable-snc points and $X$ is stable-snc in a neighbourhood of $T$. The
process terminates after finitely many iterations of Steps 2 and 3 (see Remark \ref{rem:why}), when
$(X,E)$ becomes stable-snc.
\end{proof}

\begin{remarks}\label{rem:varfunct} 
(1) The desingularization algorithm of Theorem \ref{thm:mainDEqualZero} is functorial with respect to
\'stale or smooth morphisms that preserve the 
number of irreducible components of $X$ at every point; cf. \cite[Sect.\,9]{BV}.
\smallskip

\noindent
(2) If $a \in C_j$, where $C_j \subset X_j$ is the centre of the blowing-up $\s_{j+1}$, then the component
of $C_j$ at $a$ lies in all irreducible components of $X_j$ at $a$.
\end{remarks}


\section{Characterization of stable-snc singularities of a triple}\label{subsec:char}

The remainder of the paper is devoted to an algorithmic proof our main theorem \ref{thm:maintriples} for a general
triple $(X,D,E)$. We will begin by making $(X,E)$ stable-snc, using Theorem
\ref{thm:mainDEqualZero}.
The remainder of the proof is by induction on the number of irreducible components of $X$, so we will henceforth assume
that the components of $X$ have a given ordering $X = X^{(1)}\cup\cdots\cup X^{(m)}$.  Theorem
\ref{thm:maintriples} will be functorial with respect to triples $(X,D,E)$ where the components of $X$ have a fixed ordering.

The proof involves a characterization of stable-snc points (Proposition \ref{lemmafactorssnc} below) that plays a role
similar to that played by Theorem \ref{thm:char} in the proof of Theorem \ref{thm:mainDEqualZero}, but in the
inductive setting needed here; in particular, Proposition \ref{lemmafactorssnc} involves the assumption that
$(X,D,E)$ is stable-snc after dropping the last component $X^{(m)}$ of $X$ together with the components of $D$ that
lie in $X^{(m)}$. Proposition \ref{lemmafactorssnc} will be used after reducing the main problem to the case that
$(X,E)$ is stable-snc and $D$ is a reduced divisor on $X$ with no components in $\Sing X\cup \Supp E$. Proposition \ref{lemmafactorssnc} treats points lying in at least two components of $X$ and in the support of $D$. Points lying outside the support of $D$ are already stable-snc by assumption, and points lying in only one component of $X$ can be studied using Proposition \ref{thm:char}.

The inductive proof of Theorem \ref{thm:maintriples} begins with the case that $X$ is smooth and
irreducible. In this case, stable-snc means that $D$ is snc. Snc points of a divisor can be
characterized either using the desingularization invariant \cite[Thm.\,3.4]{BDV} or (as a particular case of stable-snc) 
by Theorem \ref{thm:char} with $c_1=\ldots=c_{\kappa_X(a)}=1$, or $c_1 = 0$ if $\kappa_X(a)=1$; 
Theorem \ref{thm:maintriples} in the case that $X$ is smooth and irreducible follows from \cite[Thm.\,1.4]{BDV}
or from Theorem \ref{thm:mainDEqualZero}.

In the inductive setting of the proof of our main theorem, 
we will use a partition of the last component $X^{(m)}$ of $X$ that is similar but not identical to
the partition in Definition \ref{def:sig}.

\begin{definition}\label{def:sigmapq}
Consider $\Omega=(e,c)$, where $e \in \IN$ and $c:= (c_1,c_2,\ldots,c_n)$ with $n\leq m$ and $c_1\geq c_2\geq\ldots\geq c_n\geq0$. Assume that $(X,E)$ is stable-snc and that $D$ has no components in $\Sing X \cup \Supp E$.  Let $q \in \IN$.
We define $\Sigma_{\Omega,q}=\Sigma_{\Omega,q}(X,D)=\Sigma_{\Omega,q}(X,D,E)$ as the set of points $a \in X^{(m)}$ such that:
\begin{enumerate}
\item there are precisely $n$ different components $X^{(i_1)},\ldots,X^{(i_n)}$ of $X$ such that, for each $j$, either $X^{(i_j)} = X^{(m)}$ or $X^{(i_j)}$ contains a component of $D$ at $a$;
\item $e$ is the minimal embedding dimension of $\cup_{j=1}^{n}X^{(i_j)}$ at $a$;
\item $c_1,\ldots,c_{n}$ are the codimensions of $X^{(i_1)},\ldots,X^{(i_n)}$, respectively, in a minimal embedding
variety for $\cup_{j=1}^{n}X^{(i_j)}$ at $a$.
\item $q$ is the minimum number of components of $D$ at $a$ which lie in any one of the $X^{(i_j)}$.
\end{enumerate}
\end{definition}

As in Definition \ref{def:sig}, we list the $c_i$ in decreasing order so that the stratum
$\Sigma_{\Omega,q}$ corresponds to a value of the Hilbert-Samuel function 
(Definition \ref{def:Hpq} below), which does not depend on an ordering of the $c_i$.

\begin{example}\label{ex:char}
Consider $X:=X^{(1)}\cup X^{(2)}\cup X^{(3)}\subset\mathbb{A}^6$, where $X^{(1)}=(x_1=x_2=0)$, $X^{(2)}=(x_4=0)$ and $X^{(3)}=(x_3=0)$, and let $D=(x_1=x_2=y_1=0)+(x_3=y_1y_2=0)$. Let $a = 0$. Then $a \in \Sigma_{(6,2,1),1}$. $D$ has
two components, one in each of $X^{(1)}$ and $X^{(3)}$. The latter have codimensions $2$ and $1$, respectively, in
$\mathbb{A}^6$, which is a minimal embedding variety already for $X^{(1)}\cup X^{(3)}$.
\end{example}

\begin{definition}\label{def:Hpq}
Consider $\Omega=(e,c)$, with $c=(c_1,\ldots,c_n)$, and $q\in\mathbb{N}$, as in Definition \ref{def:sigmapq}.
Assume that $|c|+q\leq e$, where $|c| = c_1 + \cdots + c_n$.  We let $H_{\Omega,q}$ denote the Hilbert-Samuel function of the ideal
\[
\bigcap_{i=1}^{n}(x_{i,1},\ldots, x_{i,c_i},y_1\dotsm y_q)
\]
in the ring of formal power series $\IK\llbracket x_{1,1},\ldots,x_{n,c_n},
y_1,\ldots,y_{e-|c|}\rrbracket$. (See Section \ref{sec:HS}). 
\end{definition}

The $H_{\Omega,q}$ are precisely the values that the Hilbert-Samuel function of $\Supp D$
can take at stable-snc points.\medskip

See \cite[Example 5.6]{BV} for an illustration of the kind of information provided by the Hilbert-Samuel function.
The condition that the Hilbert-Samuel function of $\Supp D$ equal $H_{\Omega,q}$ at a point of $\Sigma_{\Omega,q}$ is 
necessary for stable-snc. But it is not sufficient, as shown by \cite[Example 5.6]{BV}. Additional geometric data are needed; these will be given using an ideal sheaf that is an obstruction
to stable-snc (Definition \ref{definitionJ}). This obstruction will be eliminated using ``cleaning-type''
blowings-up similar to those used
in \cite[Sect.\,7]{BV} to eliminate an analogous obstruction; see Proposition \ref{prop:desingOfJ}.

Lemma \ref{lem:HSlemma} in the following section is used in the proof of Proposition \ref{lemmafactorssnc}, 
and provides some initial control over the divisor $D$ at a point
of $\Sigma_{\Omega,q}$ where $X$ has $\geq 2$ components and the Hilbert-Samuel function has the ``correct''
value $H_{\Omega,q}$.
 
\begin{definition}\label{def:pairs}
Assume that no irreducible component of $D$ lies in $\Sing X$. Set $X^i:=X^{(1)}\cup\ldots\cup X^{(i)}$, $1\leq i\leq m$. Let $D_i$ denote the sum of all components of $D$ lying in $X^{(i)}$; i.e. $D_i$ is the divisorial part of the restriction of 
$D$ to $X^{(i)}$. We will sometimes write $D_i=D|_{X^{(i)}}$. Set $D^i:=\sum_{j=1}^iD_i$.
\end{definition}
 
\begin{definition}\label{definitionJ} \emph{Obstruction ideal.}
Assume that $X$ is
stable-snc, and that no irreducible component of $D$ lies in $\Sing X$. 
Let $J=J(X,D)$ denote the quotient ideal sheaf
 \[
 J=J(X,D):=\bigcap_{1\leq i,j\leq m}[I_{D_i}+I_{X^{(j)}}:I_{D_{j}}+I_{X^{(i)}}],
 \]
where $I_{D_i}$, $I_{X^{(j)}}$, $I_{D_{j}}$ and $I_{X^{(i)}}$ are the
ideal sheaves of $\Supp D_i$, $X^{(j)}$, $\Supp D_{j}$ and $X^{(i)}$
(respectively) in $\cO_X$. 
\end{definition}

Note that, at a point which does not lie in some component $X^{(i)}$ of $X$, all quotients involving $X^{(i)}$
in the intersection above 
are equal to $\mathcal{O}_X$ and can therefore be ignored.

An ideal sheaf defined in a similar way to $J(X,D)$ above was used in \cite{BV}. Definition \ref{definitionJ} is 
more suitable here, and in fact also simplifies the argument in \cite{BV}.

We consider decompositions $X=Y\cup T$, where $Y$, and $T$ are two closed subvarieties with no common components. The inductive characterization of stable-snc will be formulated using a $4$-tuple of the form $(Y,D,E,T)$, where $X=Y\cup T$, $(X,E)$ is stable-snc, and $D$ is a Weil divisor on $Y$ such that $(Y,D,E|_Y)$ is stable-snc.

\begin{definition}\label{def:4tuples}
We say that $(Y,D,E,T)$ is stable-snc at $a$ if there exists a Weil divisor $D_T$ on $T$ such that 
$(Y\cup T, D+D_T ,E)$ is stable-snc at $a$. The \emph{transform} of $(Y,D,E,T)$ by a sequence of admissible blowings-up
for $(X,E)$ is given by the transform of $(X,D,E)$ as in Definition \ref{not:seqembedded}.
\end{definition}

\begin{proposition}[Inductive characterization of stable-snc]\label{lemmafactorssnc}
Consider a triple $(X,D,E)$ satisfying the hypotheses of Theorem \ref{thm:maintriples} and 
let $X^{(i)}$, $i=1,\ldots,m$, denote the irreducible components of $X$ (ordered, as above). 
Assume $m\geq 2$. Let $a \in X^{m-1} \cap X^{(m)}$ (in the notation above). Then:
\begin{enumerate}
\item 
$(X,D,E)$ is stable-snc at $a$ if and only if both $(X^{m-1},D^{m-1},E, X^{(m)})$ and
$(X,D)$ are stable-snc at $a$.
\item
Suppose that
$D$ is reduced, with no irreducible component in $\Sing X$. Assume that $a$
belongs to at least two components of $D$, one in $X^{(m)}$ and the other in $X^{(i)}$, for some $i \neq m$. Then $(X,D)$ is stable-snc at $a$ if and only if
\begin{enumerate}
\item $(X^{m-1},D^{m-1},0, X^{(m)})$ is stable-snc at $a$;
\item there exist $\Omega$ and $q$ as in Definition \ref{def:sigmapq},
 such that $a\in\Sigma_{\Omega,q}(X,D)$ and $H_{\Supp D,a}
 =H_{\Omega,q}$, where $H_{\Supp D,a}$ is the Hilbert-Samuel function of $\Supp D$ at $a$;
 \item $J_{a}=\mathcal{O}_{X,a}$.
 \end{enumerate}
 \end{enumerate}
 \end{proposition}
 
 Proposition \ref{lemmafactorssnc} will be proved at the end of Section \ref{sec:HS}.

 \begin{remarks} Consider assertion (2) of the theorem. (1) If $a$ lies in $X^{(m)}$ but in no other $X^{(i)}$, then 
of course (a) is vacuous 
 and $J_a =\mathcal{O}_{X,a}$. 
In this case, Theorem \ref{thm:char} applied to $(X^{(m)},\Supp D)$ replaces Proposition
\ref{lemmafactorssnc}.
\smallskip

(2) We will use Proposition \ref{lemmafactorssnc} to remove unwanted
singularities at points lying in more than two components of $X$,
by first blowing up to get either $a \notin X^{(m)}$, or $a \in X^{(m)}$ satisfying (b), and then applying further blowings-up to get (c); 
see Section \ref{sec:desingularization of J}.
\smallskip

(3) Note that, assuming (a), $J$ as given in Definition \ref{definitionJ} coincides with the intersection 
for $i=1,\ldots,m-1$ and $j=m$.
\end{remarks}

\section{The Hilbert-Samuel function and stable simple normal crossings}\label{sec:HS}

Lemma \ref{lem:HSlemma} of this section plays an important
part in our use of the Hilbert-Samuel function to characterize stable-snc points.
We begin with the definition of the Hilbert-Samuel function and
its relationship with the diagram of initial exponents (cf.\,\cite{BMjams}).
At the end of the section, we use Lemma \ref{lem:HSlemma} to prove the
inductive characterization of stable-snc (Lemma \ref{lemmafactorssnc}).

\begin{definition}\label{def:HS}
Let $A$ denote a Noetherian local ring $A$ with maximal ideal $\fm$. The \emph{Hilbert-Samuel function} $H_A \in \IN^{\IN}$ of $A$ is defined by
\[
H_A(k) := \length \frac{A}{\fm^{k+1}}, \quad k \in \IN.
\]
If $I \subset A$ is an ideal, we sometimes write $H_{I}:=H_{A/I}$.
If $X$ is an algebraic variety and $a\in X$ is a closed point, we define
$H_{X,a}:=H_{{\cO}_{X,a}}$, where $\cO_{X,a}$ denotes the local ring of $X$ at $a$.
\end{definition}

\begin{definition}\label{definitionorderHSfunction}
Let $F,G \in \mathbb{N}^{\mathbb{N}}$. We say that $F>G$ if $F(n)\geq G(n)$, for every $n$, and  $F(m)>G(m)$, for some $m$. This relation induces a partial order on the set of all possible values for the Hilbert-Samuel functions of Noetherian local rings. 
\end{definition}
Note that $F\nleq G$ if and only if either $F>G$ or $F$ is incomparable to $G$.

Let $\wA$ denote the completion of $A$ with respect to $\fm$. Then $H_{A}=H_{\wA}$ \cite[\S24.D]{Ma}. If $A$ is regular, then we can identify $\wA$ with a ring of formal power series, $\IK\llbracket x\rrbracket$, where $x=(x_1,\ldots,x_n)$. Then  
\[
H_I(k):=\dim_{\IK}\frac{\IK\llbracket x\rrbracket}{I+\fn^{k+1}},
\]
where $\fn := (x_1,\ldots,x_n)$
is the maximal ideal of $\IK\llbracket x\rrbracket$.

If $\alpha=(\alpha_1,\ldots,\alpha_n)\in\mathbb{N}^{n}$, set $|\alpha|:=\alpha_1+\ldots+\alpha_n$. The lexicographic order of $(n+1)$-tuples, $(|\alpha|,\alpha_1,\ldots,\alpha_n)$ induces a total ordering of $\mathbb{N}^n$. Let $f\in \IK\llbracket x\rrbracket$ and write $f=\sum_{\alpha\in\mathbb{N}^n}f_\alpha x^{\alpha}$, where $x^{\alpha}$ denotes $x_1^{\alpha_1}\dotsm x_n^{\alpha_n}$. Define $\supp(f)=\{\alpha\in\mathbb{N}^n:\,f_\alpha\neq0\}$. The \emph{initial exponent} $\exp(f)$ is defined as the smallest element of $\supp(f)$. If $\al=\exp(f)$, then $f_\alpha x^{\alpha}$ is called the \emph{initial monomial} $\mon(f)$
of $f$.

\begin{definition}\label{def:diag}. Consider an ideal $I\subset K\llbracket x\rrbracket$. The \emph{initial monomial ideal}
$\mon(I)$ of $I$ denotes the ideal generated by $\{\mon(f): f\in I\}$. The \emph{diagram of initial exponents} $\cN(I)\subset\mathbb{N}^n$ is defined as
\[
\cN(I) := \{\exp(f): f \in I\setminus\{0\}\}.
\]
\end{definition}

Clearly, $\cN(I) + \IN^n = \cN(I)$. For any $\cN \subset\mathbb{N}^n$ such that $\cN=
\cN+\mathbb{N}^n$, there is a smallest set $\cV \subset \cN$ such that $\cN=\cV+\cN$; moreover, $\cV$ is finite. We call $\cV$ the set of \emph{vertices} of $\cN$. 

\begin{proposition}\label{prop:diag}
For every $k\in\mathbb{N}$, $H_{I}(k)=H_{\mon(I)}(k)$ is the number of elements $\alpha\in\mathbb{N}^n$ such that $\alpha\notin \cN(I)$ and $|\alpha|\leq k$.
\end{proposition}
\begin{proof}
See \cite[Corollary 3.20]{BMinv}.
\end{proof}

In \cite[Example 5.6]{BV}, although the intersections $D_1\cap X^{(2)}$ and $D_2\cap X^{(1)}$ are not the same, $D_2\cap X^{(1)}$ has the same components as $D_1\cap X^{(2)}$ plus some extra components (precisely, plus one extra component $(x_1=x_2=z=0)$). The following lemma shows that this is the worst that can happen when we have the correct value $H_{\Omega,q}$ of the Hilbert-Samuel function in $\Sigma_{\Omega,q}$.

\begin{lemma}\label{lem:HSlemma}
Consider $a\in\Sigma_{\Omega,q}$, where $\Omega=\left(e,(c_1,\ldots,c_m)\right)$ and $m\geq2$.
Assume that $X$ is embedded locally in a coordinate chart of a smooth variety $Z$ of minimal dimension, with a system of coordinates $\{x_{i,j}\}_{1\leq i\leq m,\ 1\leq j\leq c_i}$, $\{y_i\}_{1\leq i\leq r}$, $\{w_i\}_{1\leq i\leq n-|c|-q}$. Assume $X = V(\cap_{i=1}^m(x_{i,1},\ldots x_{i,c_i}))$.
Suppose that $D$ is a reduced divisor (so we view it as a subvariety), with no component in 
$\Sing X$, given at $a=0$ by an ideal $I_D$ of the form
\begin{equation}\label{eq:HSlemma}
I_D=\left[\bigcap_{i=1}^{m-1}(x_{i,1},\ldots x_{i,c_i})+(y_1\dotsm y_r)\right]\cap(x_{m,1},\ldots,x_{m,c_m},f).
\end{equation}
(In particular, $q$ is the minimum of $r$ and the number of irreducible factors of $f|_{(x_{m,1}=\ldots=x_{m,c_m}=0)}$\emph). 

Let $H_D$ denote the Hilbert-Samuel function $H_{I_D}$.
Then $H_D=H_{\Omega,q}$ if and only if we can choose $f$ so that $\ord\,f=q$, $r=q$ and 
$$
f\in J:=\bigcap_{i=1}^{m-1}(x_{i,1},\ldots x_{i,c_i})+(y_1\dotsm y_r)+(x_{m,1},\ldots,x_{m,c_m}).
$$
Moreover, if either $\ord\,f>q$, $r>q$ or $f\notin J$, then $H_D\nleq H_{\Omega,q}$.
\end{lemma}

\begin{remark}
It follows immediately from the conclusion of the lemma that $H_{D}\not< H_{\Omega,q}$ at a point in $\Sigma_{\Omega,q}$.
\end{remark}

\begin{proof}[Proof of Lemma \ref{lem:HSlemma}]

First we give a more precise description of the ideal $I_D$. Write $I_i:=(x_{i,1},\ldots,x_{i,c_i})$, $i=1,\ldots,m$. Let $K\subset\{1,2,\ldots,m-1\}\times\{1,2,\ldots,r\}$ denote the set of all $(i,j)$ such that $f\in I_m+I_i+(y_j)$. If $(i,j)\in K$, then
any element of $I_m+(f)$ belongs to the ideal $I_m+I_i+(y_j)$. Set 
$G:=\cap_{(i,j)\in K}(I_i+(y_j))$ and $H:=\cap_{(i,j)\notin K}(I_i+(y_j))$ 
(where the intersections are taken to be the local ring $\cO_{Z,a}$ if the index set is empty); note that these are the prime decompositions. Then any element of $I_m+(f)$ belongs to 
$\cap_{(i,j)\in K}(I_m+I_i+(y_j))=I_m+G$. Therefore we can take $f\in G$. Observe that we still have $f\notin I_i+(y_j)$ for $(i,j)\notin K$. By a computation the same as in 
\cite[Proof of Lemma 5.7]{BV}, replacing $x_i, p$ in the latter by $I_i, m$ (respectively)
here, we get:
\begin{equation}\label{HSequation}
I_D=I_m\cdot\left[H\cap G\right]+H\cdot(f).
\end{equation}

The remainder of the proof is also quite similar to the hypersurface case treated in
\cite[Proof of Lemma 5.7]{BV}, but we include it because it is not a direct translation 
as above. In particular, the diagrams of initial exponents here are more complicated. 

We can pass to the completion of $\cO_{Z,a}$ because this does not change
the Hilbert-Samuel function, the order of $f$ or ideal membership.
So we assume we are working in a formal power series ring, where $\{x_{i,j}\}_{1\leq i\leq m,\ 1\leq j\leq c_i}$, $\{y_i\}_{1\leq i\leq r}$, $\{w_i\}_{1\leq i\leq n-|c|-q}$ are the indeterminates. 
For simplicity, we
use the same notation for ideals and their generators before and after completion.

We can compute the Hilbert-Samuel function $H_D$ using the diagram of initial exponents
$\cN(I_D)$. The latter should be compared to the diagram of the ideal
$\cap_{i=1}^{m}I_i+(y_1\dotsm y_q)$, whose Hilbert-Samuel function is $H_{\Omega,q}$.
\medskip

\noindent\textbf{A.} First we show that $H_D\nleq H_{\Omega,q}$ in the following three cases:\medskip

\noindent\textbf{Case 1.} $H\neq(1)$ or $\ord\,f>q$. Then all elements of $H\cdot(f)$ have order $>q$. Moreover, all elements of
\[
I_m\cdot[G\cap H]=I_m\cdot\left(\bigcap_{i=1}^{m-1}I_i+(y_1\dotsm y_r)\right)
\]
of order less than $q+1$ have initial monomials in $\cN(\cap_{i=1}^{m}I_i)$.

It follows that, if $H\neq(1)$ (i.e., $f\notin\cap_{i=1}^{m-1}I_i+(y_1\dotsm y_r)$)
or if $\ord\,f>q$, then $H_D\nleq H_{\Omega,q}$. In fact, in $\cN(I_D)$, below degree $q+1$, we have at most the vertices of $\cN(\cap_{i=1}^{m}I_i)$, while in 
$\cN(\cap_{i=1}^{m}I_i+(y_1\dotsm y_q))$, below degree $q+1$, we have also the vertex corresponding to the monomial $y_1\dotsm y_q$.
\medskip

\noindent\textbf{Case 2.} $H=(1)$ (i.e., $f\in\cap_{i=1}^{m-1}I_i+(y_1\dotsm y_r)$), $\ord\,f=q$ and $r>q$. Then $\mon(f) \in \cap_{i=1}^{m-1}I_i+(y_1\dotsm y_r)$, after perhaps adding
an element of $I_m$ to $f$. A simple computation shows that
\[
\mon(I_D) = I_m\cdot\left(\bigcap_{i=1}^{m-1}I_i+(y_1\dotsm y_r)\right)+ (\text{mon}(f)).
\]
This follows from the fact that cancelling the initial monomial of $f$ using elements of 
$I' := I_m\cdot(\cap_{i=1}^{m-1}I_i+(y_1\dotsm y_r))$ leads to a function whose initial monomial is already in $I'$. In fact, $f\in\cap_{i=1}^{m-1}I_i+(y_1\dotsm y_r)$ but, since $\ord(f)=q$, then $\text{mon}(f)\in\cap_{i=1}^{m-1}I_i$. This means that to eliminate the initial monomial of $f$, we multiply $f$ by an element of $I_m$ and then subtract an element of $I'$. This results again in an element of $I'$, and
therefore contributes no new vertices to $\cN(I_D)$.

It follows that $H_{D}\nleq H_{\Omega,q}$. In fact, since 
$\text{mon}(f)\in\cap_{i=1}^{m-1}I_i$, there exists $b\in\cap_{i=1}^{m}I_i$ (any $b$ that is not relatively prime to $\text{mon}(f)$) such that there are points in $\cN(I_D)$ that correspond to monomials that are multiples of both $\text{mon}(f)$ and $b$, but not of $\text{mon}(f)\cdot b$. This implies that, in degree $\deg(\lcm(\mon(f),b))=q+1$, there are fewer vertices in 
$\cN(I_D)$ than $\cN(\cap_{i=1}^{m}I_i+(y_1\dotsm y_q))$; therefore $H_{D}(q+1)>H_{\Omega,q}(q+1)$.
\medskip

\noindent\textbf{B.} Secondly, we show $H_D=H_{\Omega,q}$, assuming that $H=(1)$ (i.e., $f\in\cap_{i=1}^{m-1}I_i+(y_1\dotsm y_r)$), $\ord\,f=q$ and $r=q$. The first assumption implies that
\begin{equation}\label{eq:ideal}
I_D=\bigcap_{i=1}^{m}I_i+I_m\cdot(y_1\dotsm y_q)+(f).
\end{equation}
Therefore, either $\mon(f) = y_1y_2\dotsm y_q$ or $\mon(f) \in \cap_{i=1}^{m-1}I_i$.

In both cases, by the same argument as in Case 2 above, 
\begin{equation*}
\mon(I_D) = \bigcap_{i=1}^{m}I_i+I_m\cdot(y_1\dotsm y_q)+(\mon(f)).
\end{equation*}

We want to prove that $H_{\text{mon}(I_D)}=H_{\Omega,q}$. If $\mon(f) = y_1y_2\dotsm y_q$,
then $H_{\text{mon}(I_D)}=H_{\Omega,q}$, by the definition of $H_{\Omega,q}$. 
On the other hand, if $\mon(f) \in \cap_{i=1}^{m-1}I_i$, then 
the Hilbert-Samuel function of $I'' := \cap_{i=1}^{m}I_i+(\text{mon}(f))$ is larger than 
$H_{\Omega,q}$ because, for each monomial $b$ representing a vertex of 
$\cN(\cap_{i=1}^{m}I_i)$ that is not relatively prime to $\text{mon}(f)$, the monomials that are multiples of both $\text{mon}(f)$ and $b$ are not only those that are multiples of $\text{mon}(f)\cdot b$. We will count the additional monomials (for each degree), and show that this number  
equals the number of monomials in $I_m\cdot(y_1\dotsm y_q)$ that are not already in $I''$; i.e., the number of points of 
$\cN(I''+I_m\cdot(y_1\dotsm y_q))$ additional to those
of $\cN(I'')$.

Write a representative of a vertex of $\cN(\cap_{i=1}^{m}I_i)$ that is not relatively prime to 
$\text{mon}(f)$ as $ax_{m,i}b$, where $\text{mon}(f)=ac$ and $x_{m,i}b,c$ are relatively prime. The monomials to be counted are of the form $ax_{m,i}bcM=\text{mon}(f)x_{m,i}bM$, 
for some monomial $M\notin\cap_{i=1}^{m-1}I_i$. Now, $y_1\dotsm y_q x_{m,i}M\in\text{mon}(I_D)$ has the same degree as $ax_{m,i}bcM$, but does not lie in $\cup_{i=1}^{m-1}I_i+(\text{mon}(f))$. This implies that, in each degree, $\cN(I_D)$ and 
$\cN(I''+I_m\cdot(y_1\dotsm y_q))$ have the same number of points. Therefore $H_D=H_{\text{mon}(I_D)}=H_{\Omega,q}$. This completes the proof of Lemma \ref{lem:HSlemma}.
\end{proof}

\begin{corollary}\label{cor:HSlemma}
In the setting of Lemma \ref{lem:HSlemma}, if there exists $q'$ such that $H_{\Omega,q'}\geq H_{\Supp D,a}$ at $a\in \Sigma_{\Omega,q}$, then $H_{\Omega,q'}\geq H_{\Omega,q}$.
If, moreover, $q'=q$, then $H_{\Omega,q}=H_{\Supp D,a}$.
\end{corollary}

\begin{proof}
As in the proof of Lemma \ref{lem:HSlemma}, we pass to the completion of $\cO_{Z,a}$.
We have
\begin{equation}\label{eq:formulaforID}
I_{D}=\bigcap_{i=1}^{m}I_i+I_m\cdot(y_1\dotsm y_r)+(f)\cdot H,
\end{equation}
with $H$ as in the proof of the lemma.
Recall that $r\geq q$ and $\ord\,f\geq q$. In the right hand side of \eqref{eq:formulaforID}, the first two terms are generated by monomials of degrees $m$ and $r+1$, respectively, while the last term is an ideal of order at least $q+1$. We compare $\cN(I_D)$ with $\cN(I)$, where
\begin{equation}
 I:=\bigcap_{i=1}^{m}I_i+(y_1\dotsm y_{q'})
\end{equation}
and where we assume $H_I=H_{\Omega,q'}$. Then $\cN(I)$ has the same vertices in degree 
$m$ as $\cN(\bigcap_{i=1}^{m}I_i)$, and these vertices are the same as those of
$\cN(I_D)$ in degree $m$. In addition, $\cN(I)$ has a vertex in degree $q'$. Since $H_{\Omega,q'}\geq H_{\Supp D,a}$ we have 
\begin{equation}\label{eq:forcor}
q'\geq\min(r+1,\ord ((f)\cdot H)).
\end{equation}
This implies that $H_{\Omega,q'}\geq H_{\Omega,q}$.

If, moreover, $q'=q$, then \eqref{eq:forcor} implies that $H=(1)$ and $\ord(f)=q$. As at the end of the proof of the lemma, it follows that $H_{\Supp D,a}=H_{\Omega,q}$.
\end{proof}

\begin{proof}[\bf{Proof of Proposition \ref{lemmafactorssnc}}]
In (1), the ``only if'' direction is obvious. Suppose that $(X,D)$ is stable-snc at $a$.
Then $(X,D,E)$ is stable-snc at $a$ if and only if $D|_Z + E|_Z$, where $Z$ denotes the intersection of the components of $X$ at $a$, is an snc divisor on $Z$. Since $(X,D)$ is 
stable-snc at $a$, the restriction of $D$ to $Z$ is the same as that of $D^{m-1}$. But, if
$(X^{m-1},D^{m-1},E,X^{(m)})$ is stable-snc at $a$, then $D^{m-1}|_Z + E|_Z$ is
an snc divisor.

For (2), first assume that $(X,D)$ is stable-snc at $a$. Then (a) is obvious.
The ideal of $\Supp\,D$ has the form $\cap_{i=1}^{m}(I_i+(y_1\dotsm y_q))$,
where $I_i:=(x_{i,1},\ldots,x_{i,c_i})$, $i=1,\ldots, m$,
in suitable coordinates for a minimal embedding variey $Z$ of $X$ at $a=0$ (recall that $D$ is reduced). Then (b) follows and, for (c), we compute
 \[
 J_{a}=\bigcap_{1\leq i\neq j\leq m}[(I_i+I_j+(y_1\dotsm y_q)):(I_i+I_j+(y_1\dotsm y_q)]=\mathcal{O}_{X,a}.
 \]
 
Conversely, assume the conditions (a)--(c). By (a), there is a system of coordinates $\{x_{i,j}\}_{1\leq i\leq m,\ 1\leq j\leq c_i}$, $\{y_i\}_{1\leq i\leq q}$, $\{z_i\}_{1\leq i\leq n-|c|-q}$ for $Z$ at 
$a$, in which $X^{(m)}=(x_{m,1}=\ldots=x_{m,c_m}=0)$ and $\Supp D$ is defined by the ideal
\[
I_D=\left(I_m+(f)\right)\cap\bigcap_{i=1}^{m-1}\left(I_i+\left(y_1\dotsm y_q\right)\right).
\]
By (b) and Lemma \ref{lem:HSlemma}, we can choose $f\in\cap_{i=1}^{m-1}\left(I_i+\left(y_1\dotsm y_q\right)\right)+I_m$, and therefore we can choose $f\in\cap_{i=1}^{m-1}\left(I_i+\left(y_1\dotsm y_q\right)\right)=\cap_{i=1}^{m-1}I_i+\left(y_1\dotsm y_q\right)$. Write $f$ in the form $f=g_1+y_1\dotsm y_q g_2$, where $g_1\in\cap_{i=1}^{m-1}I_i$. Then
\begin{equation*}
\begin{aligned}
 J_{a}&=\bigcap_{i=1}^{m-1}[(I_m+I_i+(f)):(I_m+I_i+y_1\dotsm y_q)]\\
   &=\bigcap_{i=1}^{m-1}[(I_m+I_i+(y_1\dotsm y_q g_2)):(I_m+I_i+(y_1\dotsm y_q))]\\
   &=\bigcap_{i=1}^{m-1}\left(I_m+I_i+(g_2)\right).
 \end{aligned}
\end{equation*}
Since no component of $D$ lies in $\Sing X$, then $g_2\notin I_m+I_i$, $i=1,\ldots,m-1$. Therefore,
$J_{a}=I_m+ (g_2)+\cap_{i=1}^{m-1}I_i$.

The condition $J_{a}=\mathcal{O}_{Y,a}$ means that $g_2$ is a unit. Then
\begin{align*}
I_D&=\left[\bigcap_{i=1}^{m-1}I_i+(y_1\dotsm y_qg_2)\right]\cap(I_m+(f))\\
 &=\left[\bigcap_{i=1}^{m-1}I_i+(g_1+y_1\dotsm y_qg_2)\right]\cap(I_m+(f))\\
 &=\left[\bigcap_{i=1}^{m-1}I_i+(f)\right]\cap(I_m+(f)).
\end{align*}
Since no component of $D$ lies in $\Sing X$, then $f\notin I_m+I_i$, for every $i=1,\ldots,m-1$. Therefore,
$I_D=\cap_{i=1}^{m}I_i+(f)$.

By Lemma \ref{lem:HSlemma}, since $a\in\Sigma_{\Omega,q}$, $\ord\,f=q$. It follows that
$f|_{V(I_m)}$ is a product $f_1\dotsm f_q$ of $q$ irreducible factors each of order one. For 
each $i=1,\ldots,q$, set $A_i:=\{(j,k): f_i\in I_j+(y_k)|_{V(I_m)},\ j\leq m-1,\ k\leq q\}$. Then $f_i\in\cap_{(j,k)\in A_i}(I_j+(y_k))|_{V(I_m)}$, where the intersection is understood to be the entire local ring if $A_i = \emptyset$. Note that $\cup_i A_i=\{(j,k):\ j\leq m-1,\ k\leq q\}$,
since $f\in\cap_{i=1}^{m-1}(I_i+(y_1\dotsm y_q))$.

We will extend each $f_i$ to a regular function on $Z$ (still denoted $f_i$) preserving the condition that $f_i\in\cap_{(j,k)\in A_i}(I_j+(y_k))$. In fact, $\cap_{(j,k)\in I_i}(I_j+(y_k))|_{V(I_m)}$ is generated by a finite set of monomials $\{m_r\}$ in the
$x_{\alpha,\beta}|_{V(I_m)}$ and $y_k|_{V(I_m)}$. Then $f_i$ is a combination $\sum m_ra_r$. So we can get an extension of $f_i$ as desired, using arbitrary
extensions of the $a_r$ to regular functions on $Z$. This means we can assume that $f=f_1\dotsm f_q\in\cap_{i=1}^{m-1}I_i+(y_1\dotsm y_q)$ (using the extended $f_i$).

Since $f|_{V(\sum_{i=1}^{m-1}I_i)}=y_1\ldots y_q g_2$, where $g_2$ is a unit, it follows that
$f=y_1\ldots y_q g_2$ mod $\sum_{i=1}^{m-1}I_i$, where $g_2$ is a unit. Since 
$I_D=\cap_{i=1}^{m}I_i+(f)$, it remains to check only that $\{x_{i,j}\}_{1\leq i\leq m,\ 1\leq j\leq c_i}$, $f_1,\ldots,f_q$ are part of a coordinate system. We can pass to the completion of
$\cO_{Z,a}$, which we identify with a ring of formal power series in variables
including $\{x_{i,j}\}_{1\leq i\leq m,\ 1\leq j\leq c_i}$, $\{y_i\}_{1\leq i\leq q}$. It is enough to prove that the images of the
$f_i$ and $x_{i,j}$ in $\hat{m}/\hat{m}^2$ are linearly independent, where $\hat{m}$ is the maximal ideal of the completed local ring. 
If we put $x_{i,j}=0$ for every $(i,j)$ in the power series representing each $f_i$ we get
\[
(f_1\dotsm f_q)|_{V(\sum_{i=1}^{m}I_i)}=y_1\dotsm y_q.
\]
This means that, after reordering the $f_i$, each $f_i|_{V(\sum_{i=1}^{m}I_i)}\in(y_i)$, and the desired conclusion follows.
 \end{proof}

\section{Algorithm for the main theorem}\label{sec:maintheorem}
In this section we prove Theorem \ref{thm:maintriples}. The proof will depend on the results 
given in Sections \ref{sec:desatsinglocusofX} and \ref{sec:fixingmultiplicities} following. We divide the proof into several steps or subroutines each of which specifies certain blowings-up.
\medskip

\noindent
\textbf{Step 1.} \emph{Make $(X,E)$ stable-snc.} This is an application of Theorem \ref{thm:mainDEqualZero}. The blowings-up involved preserve stable-snc singularities of $(X,E)$ and therefore also of $(X,D,E)$. As a result of Step $1$,
we can assume that $(X,E)$ is stable-snc.
\medskip

In the following steps, all blowings-up will be both admissible and snc with respect to $X$,
to preserve the property that $(X,E)$ is stable-snc.
\medskip

\noindent
\textbf{Step 2.} \emph{Remove irreducible components of $D$ lying in $\Sing X$ or $\Supp E$.} Given a triple $(X,D,E)$, consider the union $\mathcal{Z}$ of the supports of the (irreducible) components of $D$ lying in $\Sing X\cup\Supp E$. Any such component is a component
either of the intersection of two components of $X$, or of the intersection of a component
of $X$ and a component of $E$. Therefore, $\mathcal{Z}$ is snc, in general with components of different dimensions. Blowings-up as needed can simply be given by the usual desingularization of $\mathcal{Z}$, followed by blowing up the final strict transform.

The point is that, locally, there is a smooth ambient variety, with coordinates $(x_1,\ldots,x_p,\ldots,x_n)$ in which each component of $\mathcal{Z}$ is of the form $(x_{i_1}=\ldots=x_{i_k}=0)$, $i_1<\ldots<i_k\leq p$. Let $C$ denote the set of irreducible components of intersections of arbitrary subsets of components of $\mathcal{Z}$. Elements of $C$ are partially ordered by inclusion, and are
snc with respect to $X$ and $E$. 
Desingularization of $\mathcal{Z}$ involves blowing up elements of $C$ starting with the smallest, until all components of $\mathcal{Z}$ are separated. Then blowing up the final (smooth) strict transform removes all components of $\mathcal{Z}$.

As a result of Step 2, we can assume that no component of $D$ lies in
$\Sing X$ or in $\Supp E$.
\medskip

\noindent
\textbf{Step 3.} \emph{Make $(X,D_\text{red},E)$ stable-snc} (i.e., transform $(X,D,E)$ by the blowings-up needed to make $(X,D_\text{red},E)$ stable-snc). The algorithm for Step $3$ is given following Step $4$ and the paragraph on functoriality below.
\medskip

We can therefore now assume that $(X,D_{\text{red}},E)$ is stable-snc and that $D$ has no irreducible components in $\Sing X$ or $\Supp E$.
\medskip

\noindent
\textbf{Step 4.} \emph{Make $(X,D,E)$ stable-snc.} A simple combinatorial argument for Step $4$ will be given in Section \ref{sec:fixingmultiplicities}. This completes the algorithm.
\medskip

\noindent
\textbf{Functoriality.} (See also \cite[Sect.\,9]{BV}.) The steps above involve several
applications of the general desingularization algorithm. Beginning with a local \'etale
invariant $\io$ (e.g., the Hilbert-Samuel function), the centres of blowing up are determined 
by a corresponding \'etale invariant $\inv_{\io}$ defined recursively over a sequence
of admissible blowings-up. The monomial marked ideals used in cleaning (Section 4) are 
\'etale-invariant. The obstruction ideal $J(X,D)$ (Section 6) is an invariant of \'etale
morphisms preserving the number of irreducible components of $X$ at every point
(see Remark \ref{rem:functJ}).
The functoriality assertion of the theorem follows because the blowing-up sequence 
given by the four steps above depends, at a given point, only on the
preceding objects and the desingularization invariant, as well as the number of components of 
$X$ and $D$, and their codimensions in a local minimal embedding variety.
\medskip

\noindent
\textbf{Algorithm for Step 3.} The input is a triple $(X,D,E)$, where $(X,E)$ is stable-snc, $D$ is reduced and no irreducible component of $D$ lies in $\Sing X \cup \Supp E$. We will argue by induction on the number of components of $X$. Since $D$ is reduced, we make no distinction between $D$ and $\Supp D$. The algorithm for Step 3 is given in the proof of Theorem \ref{thm:fortriples} below, applied to the $4$-tuple $(X,D,E,\emptyset)$.
 
\begin{theorem}\label{thm:fortriples}
Assume that $(X,D,E,Y)$ is a $4$-tuple as in Definition \ref{def:4tuples}, such that 
$(W:=X\cup Y,E)$ is stable-snc, and $D$ is a reduced Weil divisor on $X$ 
with no component in 
$\Sing W \cup \Supp E$. Then there is a morphism $\tau :W'\rightarrow W$ given by a 
composite of admissible smooth blowings-up whose centres are snc with respect
to $W$, such that:
\begin{enumerate}
\item Each blowing-up is an isomorphism over the stable-snc points of its target $4$-tuple.
\item The transform $(X',D',E',Y')$ of $(X,D,E,Y)$ by $\tau$ is everywhere stable-snc.
\end{enumerate}
\end{theorem}

\begin{proof}
The proof is by induction on the number of components $m$ of $X$. We use the notation of Definitions \ref{def:pairs} and \ref{def:4tuples}.
\medskip

\noindent
\emph{Case $m=1$.} For $m=1$ (i.e., $X=X^{(1)}$), we apply Theorem \ref{thm:mainDEqualZero} to $(D \cup Y|_{X},E|_X)$, and end up with $(D'\cup Y'|_{X'},E'|_{X'})$ stable-snc. Since $D$ is a divisor on $X$, then $D'$ is a divisor on $X'$, and we have 
$(X',D',E',Y')$ stable-snc. All centres of blowing up involved are snc with respect to $W = X\cup Y$, by
Remarks \ref{rem:varfunct}(2).
\medskip

\noindent
\emph{General case.} The sequence of blowings-up will depend on the ordering of the components $X^{(i)}$ of $X$.\medskip

\emph{By induction, we can assume} that $(X^{m-1},D^{m-1},E,X^{(m)}\cup Y)$ is stable-snc. We want to construct a sequence of admissible blowings-up after which the transform
$(X',D',E',Y')$ of $(X,D,E,Y)$ is stable-snc. For this purpose, we only have to remove the unwanted singularities of $D$ in $X^{(m)}$. 

\medskip\noindent
\textbf{A.} We will first reduce to the case that $(X,D,E,Y)$ is stable-snc at every point of 
$X^{m-1}$. For this purpose,
we use the partition of $X^{(m)}$ by the sets 
$\Sigma_{\Omega,q}=\Sigma_{\Omega,q}(X,D)$ (see Definition \ref{def:sigmapq}). 
Clearly, the $\Sigma_{\Omega,q}$ with $r_{\Omega}:=r \geq 2$, where
$\Omega=(e,c_1,\ldots,c_r)$, form a partition of $X^{(m)} \cap X^{m-1}$.

We use the ordering of the set of Hilbert-Samuel functions $H_{\Omega,q}$
(Definitions \ref{definitionorderHSfunction} and \ref{def:Hpq}) to order the set of tuples 
$(\Omega,q)$ and thus the strata 
$\Sigma_{\Omega,q}$ (Definition \ref{def:sigmapq}).

\begin{definition}\label{def:orderingSigmaOmegaQ}
 We say that $(\Omega_1,q_1)\geq(\Omega_2,q_2)$ and also that $\Sigma_{\Omega_1,q_1}\geq\Sigma_{\Omega_2,q_2}$, if 
 $(\Omega_1,H_{\Omega_1,q_1})\geq(\Omega_2,H_{\Omega_2,q_2})$ in the lexicographic order, where 
we compare $\Omega_1$, $\Omega_2$ also lexicographically, and $H_{\Omega_1,q_1}$, $H_{\Omega_2,q_2}$ by 
Definition \ref{definitionorderHSfunction}.
\end{definition}

The order above corresponds to that in which we will eliminate the non-stable-snc points from the strata $\Sigma_{\Omega,q}$, $r_{\Om} \geq 2$. 
\medskip

Clearly for all $\Omega$ and $q$, the closure
$\overline{\Sigma}_{\Omega,q}$ of $\Sigma_{\Omega,q}$ has the property
\begin{equation}\label{eq:closureofSigma}
\overline{\Sigma}_{\Omega,q}\subset\bigcup_{(\Omega',q')\geq (\Omega,q)}\Sigma_{\Omega',q'}.
\end{equation}

\begin{definition}\label{def:monotone}\label{def:K}
Let $\mathcal{M}$ denote the set of all possible values of $(\Omega,q)$, and let
$\mathcal{M}(X,D):=\{(\Omega,q)\in\mathcal{M}:\ \emptyset \neq \Sigma_{\Omega,q}(X,D)
\subset X^{(m)} \cap X^{m-1}\}$. Let $K(X,D)$ denote the set of maximal elements of
$\mathcal{M}(X,D)$.
\end{definition}

Note that $K(X,D)$ consists only of incomparable pairs $(\Omega,q)$, and that,
after an admissible blowing-up $\s$, all points of $\s^{-1}(a)$,
where $a \in \Sigma_{\Omega,q}$, lie in strata $\leq \Sigma_{\Omega,q}$.

We apply Proposition \ref{prop:desingOfJ} of the following section to construct a
morphism $X'\rightarrow X$ given by a sequence of admissible blowings-up
such that $(X',D',E',Y')$ is stable-snc on the strata $\Sigma_{\Omega,q}(X',D')$, where 
$(\Omega, q) \in K(X,D)$,

Let $U' := X' \setminus \bigcup_{(\Omega,q)\in K(X,D)}\Sigma_{\Omega,q}(X',D')$. By 
\eqref{eq:closureofSigma}, $U'$ is open. Clearly, $\mathcal{M}(U',D'|_{U'})=\mathcal{M}(X',D')\setminus K(X,D)$. The set $\text{Fin}(\mathcal{M})$ of finite subsets of $\mathcal{M}$, ordered by inclusion, is a partially ordered set in which every nonempty subset has a minimal element. 
We can therefore assume by induction on $\text{Fin}(\mathcal{M})$ that 
$(U',D'|_{U'},E',Y')$ is stable-snc at every point in $X^{m-1}$. The blowings-up 
involved have centres that are nowhere 
stable-snc and are, therefore, closed not only in $U'$ but also in $X'$.

\medskip\noindent
\textbf{B.} Under the assumption that $(X,D,E,Y)$ is stable-snc at every point of 
$X^{m-1}$, we complete the proof as follows: Let 
$U = X^{(m)}\setminus X^{m-1}$. We apply Theorem \ref{thm:mainDEqualZero} to 
$(D|_U\cup Y|_U,E)$ (regarding $D|_U\cup Y|_U$ 
as a subvariety of the smooth variety $U$), to get $(D'|_{U'}\cup Y'|_{U'},E')$ stable-snc. Since $D|_U$ is a divisor on $U$, then $D'|_{U'}$ is a divisor on $U'$. Therefore,
$(U',D'|_{U'},E'|_{U'},Y'|_{U'})$ is stable-snc. The centres of blowing up involved contain no
stable-snc points. Since $(X,D,E,Y)$ is stable-snc at every point of $X\setminus U$ and the stable-snc locus is open, these centers are closed not only in $U$ but also in $X$.
\end{proof}


\section{Desingularization at the singular locus of $X$}\label{sec:desatsinglocusofX}
In  this section, we complete the proof of Theorem \ref{thm:fortriples} by showing
how to eliminate non-stable-snc singularities from the strata $\Sigma_{\Omega,q}$ with 
$r_{\Omega}\geq2$. We recall that these strata consist of points belonging to 
at least $2$ components of $D$ which lie in different components of $X$. We will
use the notation of Section \ref{sec:maintheorem}.

\begin{proposition}\label{prop:desingOfJ}
Let $(X,D,E,Y)$ denote a $4$-tuple as in Definition \ref{def:4tuples}, satisfying the
hypotheses of Theorem \ref{thm:fortriples}. Assume that
$(X^{m-1},D^{m-1},E,X^{(m)}\cup Y)$ is stable-snc. 
Then there is a sequence of admissible smooth blowings-up whose centres are snc
with respect to $W = X\cup Y$, such that:
 \begin{enumerate}
 \item each centre of blowing-up contains only non-stable-snc  points;
 \item the transform $(X',D',E',Y')$ of $(X,D,E,Y)$ by the blowing-up sequence
 is stable-snc at all points of the strata $\Sigma_{\Omega,q}(X',D')$,
 where $(\Omega,q)\in K(X,D)$.
 \end{enumerate}
 \end{proposition}

The proof will involve several lemmas. We will use the assumptions of 
Proposition \ref{prop:desingOfJ} throughout the section.

Let $a\in X$. There is a minimal smooth local embedding variety $Z$ of 
$X$ at $a$, with a system of coordinates $\{x_{i,j}\}_{1\leq i\leq m,\ 1\leq j\leq c_i}$, 
$\{y_k\}_{1\leq k\leq q}$, $\{w_l\}_{1\leq l\leq n-|c|-q}$, $|c| = c_1 + \cdots c_m$, 
in which $a=0$ and
\begin{align*}
 X&=X^{(1)} \cup\ldots\cup X^{(m)},\\
D&=D_1+\ldots+D_m,
\end{align*}
where $X^{(i)}=(x_{i,1}=\ldots=x_{i,c_i}=0)$, $i=1,\ldots,m$, 
$D_i=(x_{i,1}=\ldots=x_{i,c_i}=y_1\dotsm y_q=0)$, $i=1,\ldots,m-1$, and 
$D_m=(x_{m,1}=\ldots=x_{m,c_m}=f=0)$, for some $f\in\mathcal{O}_{Z,a}$. 
This notation will be used throughout the section.

\subsection{Reduction of the obstruction ideal  $J(X,D)$ to $\mathcal{O}_{X}$.}\label{sec:desingularization of J} Recall Definition \ref{definitionJ} and Proposition 
\ref{lemmafactorssnc}.

Since $(X^{m-1},D^{m-1},0,X^{(m)})$ is stable-snc, $\cosupp J \subset X^{(m)}\cap X^{m-1}$
(where $\cosupp J := \supp\, \cO_X/J$). 
Let $W_1,\ldots,W_s$ denote the irreducible components of $(X^{m-1}\cup Y)|_{X^{(m)}}$, and let 
$D(1),\allowbreak D(2),\ldots,D(q)$  denote the restrictions to $X^{(m)}\cap X^{m-1}$ 
of the components 
of $D_i$, for any given $i=1,\ldots,m-1$ (the definition is independent of such $i$). 
Let $\cH_i \subset \cO_{X^{(m)}}$ denote the ideal of $D(i) \subset X^{(m)}$, $i=1,\ldots,q$,
and let $\cK_j \subset \cO_{X^{(m)}}$ denote the ideal of $W_j$, $j=1,\ldots,s$.

Consider the marked ideal
\begin{equation} \label{eq:MarkedIdealForJ}
\ucI := (X^{(m)},X^{(m)},E|_{X^{(m)}},J+ \cH + \cK,1),
\end{equation}
where $\cH := \sum_{i=1}^q \cH_i$ and $\cK = \sum_{j=1}^s \cK_j$.
We can use desingularization of the marked ideal $\ucI$ (treating $(\cH + \cK, 1)$
as a ``boundary''; cf. Section \ref{sec:inv}) to desingularize $(J,1)$ after perhaps moving
the $D(i)$, $i=1,\ldots,q$, and $W_j$, $j=1,\ldots,s$, away from $\cosupp (J,1)$.
The blowings-up involved are admissible for
$(X,D,E,Y)$, and snc with respect to $W= X\cup Y$ (since the boundary includes $\cK$). The final transform $J(X,D)' = \cO_{X'}$. It is not necessarily true, however, that 
$J(X,D)'=J(X',D')$, so we do not necessarily have $J(X',D')=\mathcal{O}_{X'}$. 
Additional ``cleaning'' blowings-up (given by Lemma \ref{removingJpequals2}) will be needed.

\begin{example}
Consider $X=X^{(1)}+X^{(2)}=(x_1=0)\cup(x_2=0)$ and $D_1+D_2=(x_1=y=0)+(x_2=x_1+yzw=0)$. 
Then $J(X,D)=(x_1,x_2,zw)$. The desingularization algorithm for $J$ first blows up 
$(x_1=x_2=z=w=0)$. In the $z$-chart, we get $X'=(x_1x_2=0)$ and 
$D'=(x_1=y=0)+(x_2=x_1+yzw=0)$. Then the desingularization of $J$ is completed by blowing up 
$(x_1=x_2=w=0)$. In the $w$-chart we have $X''=(x_1x_2=0)$ and 
$D''=(x_1=y=0)+(x_2=x_1+yz=0)$. Note that $J(X'',D'')=(x_1,x_2,z)\neq (1)=J(X,D)''$. Since 
$z=0$ is now a component of the exceptional divisor, we can blow up with center 
$X^{(1)}\cap X^{(2)}\cap(z=0)$. After this ``cleaning" blowing-up, we have 
$X'''=(x_1x_2=0)$, $D'''=(x_1=y=0)+(x_2=x_1+y)=(x_1=x_1 +y=0)+(x_2=x_1+y)$, and
$J(X''',D''')=(1)$; in particular $(X''', D''')$ is stable-snc.
\end{example}

\begin{lemma}\label{claimcasepequal21}
Consider the morphism $X' \to X$ given by the desingularization sequence (beginning
with that of \eqref{eq:MarkedIdealForJ} above.
Then
\begin{equation}\label{eq:inclusion}
J(X',D')\subset J(X,D)'.
\end{equation}
Moreover, if $J(X,D)'=\mathcal{O}_{X'}$ and $a'\in X'$, then
\[
J(X',D')_{a'}=\bigcap_{1\leq i\neq j\leq m} \left(I_i+I_j+(u^{\alpha_{i,j}})\right),
\]
where $u^{\alpha_{i,j}}$ are monomials in generators $u_p$ 
of the ideals of the components of the exceptional divisor of $X' \to X$.
 \end{lemma}
 
\begin{remark} By \eqref{eq:inclusion}, if $J(X,D)'\allowbreak\neq\allowbreak\mathcal{O}_{X'}$, then $J(X',D')\neq\mathcal{O}_{X'}$. Therefore, by Lemma \ref{lemmafactorssnc}, we never blow-up stable-snc points of the transforms of $(X,D)$ while desingularizing $J(X,D)$.
\end{remark}

 \begin{proof}[Proof of Lemma \ref{claimcasepequal21}]
It is enough to prove the lemma for one of the ``factors''
 $[I_{D_i}+I_{X^{(j)}} : I_{D_j}+I_{X^{(i)}}]$ of $J$. The proof is then the same as 
 that of \cite[Proof of Lemma 7.3]{BV}, replacing $x_i$ in the latter by $I_i$ here.
 \end{proof}

\begin{lemma}\label{removingJpequals2}
Consider the transform $(X',D',E',Y')$ of $(X,D,E,Y)$ by the desingularization 
sequence above. Then:
\begin{enumerate}
\item For every $(\Omega,q)$, $\Sigma_{\Omega,q}(X',D')$ lies in the inverse image of 
$\Sigma_{\Omega,q}(X,D)$.
\item
Let $a'\in X'$. Then
$$
J(X',D')_{a'} =  \bigcap_{1\leq i\neq j\leq m}(I_i+I_j+(u^\alpha_{i,j})), 
$$
where each $I_i$ denotes the ideal of the component $X^{(i)\prime}$ of $X'$, and 
the $u^{\alpha_{i,j}}$ are monomials in generators $u_p$ of the ideals of the components of
$E'$. Thus the variety $V(J(X',D'))$ consists of certain components of intersections of pairs of
components of $X'$ and components of $E'$.
\item After finitely many blowings-up of components of $V(J(X',D'))$ (and its successive 
transforms), the transform $(X'',D'')$ of $(X,D)$ satisfies $J(X'',D'')\allowbreak 
=\mathcal{O}_{X''}$. (For functoriality, the components to be blown up can be chosen 
according to the order on the components of $E$.)
\end{enumerate}
\end{lemma}

 \begin{proof}
(1) has already been remarked in the previous section. (2) and (3) can be proved
in the same way as the corresponding assertions of 
\cite[Lemma 7.5]{BV} replacing $x_i$ in the latter by $I_i$ here, and multiples of $x_i$ by 
linear combinations with coefficients in $\cO_X$
of the $x_{i,j}$ here. (2) follows from the second assertion of 
Lemma \ref{claimcasepequal21} and, for (3), we can directly compute the effect of the 
blowings-up.
 \end{proof}

\begin{remark}\label{rem:functJ}
The desingularization algorithm of Theorem \ref{prop:desingOfJ} is functorial with respect to 
\'etale morphisms that preserve the number of irreducible components at every point,
since $J$ has an \'etale-invariant meaning and the algorithms involved in desingularizing
$J$ and in cleaning are controlled by \'etale invariants.
\end{remark}

\subsection{Simplification of $\Supp D$.}\label{sec:scrambleddesingularization}
In order to prove Proposition \ref{prop:desingOfJ} above, we need to construct a blowing-up
sequence that will allow us to decrease and control the Hilbert-Samuel function on the strata 
$\Sigma_{\Omega,q}$, where $(\Omega,q)\in K(X,D)$. We can use the desingularization of 
$\Supp D$ to decrease the Hilbert-Samuel function, but we will blow up only certain 
irreducible components of the centres prescribed by this desingularization, in a convenient way.

At every point $a\in X^{(m)}$, we introduce the invariant $\io(a) = (e(a),c(a), H_{\Supp D,a})$,
where $(e(a),c(a)) = (e,c)$ is defined as in Definition \ref{def:sigmapq}. The set of values of
this invariant is partially ordered, lexicographically (using the partial ordering of the set
of Hilbert-Samuel functions given by Definition \ref{definitionorderHSfunction}). 

Clearly, the invariant $\io = ((e,c),H_{\Supp D})$ is upper-semi-continuous on $X^{(m)}$. 
Since $(e,c)$ is constant
on $\left\{x\in\Supp D :\ H_{\Supp D,x} = H_{\Supp D,a}\right\}$ near $a$ (i.e., on the
cosupport of a presentation of $H_{\Supp D}$ at $a$), a presentation of the Hilbert-Samuel function 
of $\Supp D$ at $a$ is also a presentation of the invariant $\io$. In particular, we can
extend $\io$ to a desingularization invariant $\inv_{\io}$.

The centres of blowing up involved in the desingularization algorithm for $\inv_{\io}$
are locally the same as in the standard desingularization
algorithm, corresponding to the invariant determined by $H_{\Supp D}$, but the use 
of $\io$ instead of $H_{\Supp D}$ means that, globally, the centres may have components
that are blown up in a different order. 

Given $a \in X^{(m)}$, $\io$ admits a presentation of the form 
$\ucI = (X^{(m)},X^{(m)},0,\cI,d)$ at $a$. We will consider the desingularization invariant
$\inv$ and desingularization algorithm determined by this presentation of $\io$, treating
the restrictions to $X^{(m)}$ of the components of $E$ and the remaining components
of $W = X\cup Y$ as a ``boundary'' $\ucB$ (even though the latter are not necessarily
codimension one in $X^{(m)}$. In other words, we let $\ucB$ denote the marked ideal
$(X^{(m)},X^{(m)},0,\cB,1)$, where $\cB$ denotes the sum of the ideals on $X^{(m)}$ of
the components of $E$ and the components of $W\setminus X^{(m)}$, and we consider
the desingularization algorithm given locally by desingularization of the marked ideal
$\ucI + \ucB$. The effect of the algorithm is to decrease $\io$ after perhaps moving
the components of $E$ and $W\setminus X^{(m)}$ away from $\Supp D$. The blowings-up
involved are admissible for $(X,D,E,Y)$ and snc with respect to $W$.

\begin{proposition}\label{prop:FixHSfunction}
Given $(X,D,E,Y)$ as in Proposition \ref{prop:desingOfJ}, there is a sequence of admissible 
blowings-up $(X',D',E',Y')\rightarrow (X,D,E,Y)$, with centers snc with respect to
$W = X\cup Y$ and containing no stable-snc points, such that for every $\Sigma_{\Omega,q}\in K(X,D)$ and $a\in\Sigma_{\Omega,q}(X',D')$, $H_{\Supp D,a} = H_{\Omega,q}$.
\end{proposition}

\begin{lemma}\label{lem:centersinsigmapq}
Let $C$ be an irreducible smooth subvariety of $\Supp D$. Given $(\Om,q)$, suppose that
$\io = (\Om, H_{\Om,q})$ at every point of $C$. If 
$C \cap \Sigma_{\Omega,q} \neq \emptyset$, then $C\subset\Sigma_{\Omega,q}$.
\end{lemma}

\begin{proof}
Let $a\in C \cap \Sigma_{\Omega,q}$. Since the $H_{\Supp D}$ is constant on $C$, $a$ has a neighbourhood $U\subset C$ such that each point of $U$ lies in
precisely those components of $D$ containing $a$.  Therefore, $U\subset\Sigma_{\Omega,q}$. Since the closure of $\Sigma_{\Omega,q}$ lies in the union of the $\Sigma_{\Omega',q'}$ with $(\Omega',q')\geq(\Omega,q)$, any $b \in C \setminus U$ belongs to $\Sigma_{\Omega',q'}$, for some $(\Omega',q')\geq(\Omega,q)$. Moreover, $\Omega'=\Omega$, since $\iota$ is constant on $C$. Thus $H_{\Supp D,b} = H_{\Omega,q} < H_{\Omega,q'}$. But, by Corollary \ref{cor:HSlemma}, the Hilbert-Samuel function cannot be $< H_{\Omega,q'}$ on
$\Sigma_{\Omega,q'}$ . Therefore $b \in \Sigma_{\Omega,q}$.
\end{proof}

\begin{proof}[Proof of Proposition \ref{prop:FixHSfunction}]
We consider the desingularization algorithm preceding Lemma \ref{lem:centersinsigmapq}, 
but will blow up only certain
components of the centres of blowing-up involved in the algorithm. The centres of
blowing up given by the algorithm are 
the maximum loci of $\inv$. The maximum locus of $\inv$ includes all maximal values of $\io$. The 
maximum locus of $\inv$ can be written
as a disjoint union $A\cup B$ in the following way: $A$ is the union of those components of the maximum locus containing no stable-snc points, and $B$ is the union of the remaining components. Thus $B$ is
the union of those components of the maximum locus of $\inv$ with generic point stable-snc. Each component of $B$ has Hilbert-Samuel function $H_{\Omega,q}$, for some $(\Omega,q)$, and lies in the corresponding $\Sigma_{\Omega,q}$ by Lemma \ref{lem:centersinsigmapq}.

In each year $j$ of the blowing-up history, write $A=A_j$, $B=B_j$. We will blow up with
centre $A_j$ only. Then $\inv$ decreases in the preimage
of $A_j$. In the following year $j+1$, $B_{j+1}$ may acquire new components in addition
to those of $B_j$, but eventually $A_k = \emptyset$. So we reduce to the case that
$A=\emptyset$.

\begin{lemma}\label{claim:claimAempty}
Suppose $A = \emptyset$. If $(\Omega,q)\in K(X,D)$, then 
$H_{\Supp D, a} = H_{\Omega,q}$, for all $a \in \Sigma_{\Omega,q}$.
\end{lemma}

\begin{proof}
Let $a\in\Sigma_{\Omega,q}$, where $(\Omega,q)\in K(X,D)$. Set $H = H_{\Supp D, a}$.
Assume that $H \neq H_{\Omega,q}$. Recall that, for every $b \in B$, 
$H_{\Supp D, b} = H_{\Omega',q'}$ for some $(\Omega', q')$,  and $b \in \Sigma_{\Omega',q'}$. Therefore $a\notin B$, so that $\inv(a)$ is not maximal. Thus there exists $b\in B$ such
that  $\io(b) = (\Omega',H_{\Omega',q'})$ and $(\Omega',H_{\Omega',q'}) \geq (\Omega,H)$, for some $(\Omega',q')$, and $b\in\Sigma_{\Omega',q'}$. If $\Omega' > \Omega$ then 
$(\Omega',q')>(\Omega,q)$; this contradicts $(\Omega,q)\in K(X,D)$. If $\Omega'=\Omega$ then, by Corollary \ref{cor:HSlemma}, $H_{\Omega',q'} \geq H_{\Omega,q}$. 
If $H_{\Omega',q'} > H_{\Omega,q}$, then $(\Omega',q')>(\Omega,q)$, again contradicting
$(\Omega,q)\in K(X,D)$. If $H_{\Omega',q'}= H_{\Omega,q}$, then $H=H_{\Omega,q}$,
by Corollary \ref{cor:HSlemma}, as desired.
\end{proof}

Lemma \ref{claim:claimAempty} finishes the proof of Proposition \ref{prop:FixHSfunction}.
\end{proof}

 \begin{proof}[Proof of Proposition \ref{prop:desingOfJ}]
We first reduce to the case $J=\mathcal{O}_X$, using Lemma \ref{removingJpequals2}.
The proof then has two steps:
\begin{enumerate}
\item We apply Proposition \ref{prop:FixHSfunction} to make $H_{\Supp D, a} = H_{\Omega,q}$, 
for all $a \in \Sigma_{\Om,q}$ and all $(\Omega,q)\in K(X,D)$.
\item We use Lemma \ref{removingJpequals2} to reduce to $J=\mathcal{O}_X$.
\end{enumerate}

The initial reduction to $J=\mathcal{O}_X$ is for the purpose of functoriality: The centres 
of blowing up involved in desingularization of $J$ may include points outside the strata of $K(X,D)$. Therefore, on an open set $U$ outside the strata of $K(X,D)$, the centres of
blowing up from desingularization of $J$ (in Step (2), for example)
may play a role when applying Step (1) for $K(U,D|_U)$ (in the inductive step of Case B in the
proof of Theorem \ref{thm:fortriples}).

After Step (1), $H_{\Supp D, a} = H_{\Omega,q}$, 
for all $a \in \Sigma_{\Om,q}$ and all $(\Omega,q)\in K(X,D)$. Then, by Lemma \ref{lem:HSlemma}, 
at each $a\in\Sigma_{\Omega,q}$, where $(\Omega,q)\in K(X,D)$, we have $I_{D_m}+I_{D^{m-1}}=I_{D^{m-1}}+I_m$, where $I_{D_m}$, $I_{D^{m-1}}$ and $I_m$ are the ideals of $D_m$, $D^{m-1}$ and $X^{(m)}$, respectively. In the notation of Lemma \ref{lem:HSlemma}, $I_{D_m}=(x_{m,1},\ldots,x_{m,c_m},f)$, $I_{D^{m-1}}=\cap_{i=1}^{m-1}(x_{i,1},\ldots,x_{i,c_i})+(y_1\dotsm y_r)$ and $I_m=(x_{m,1},\ldots,x_{m,c_m})$, and the lemma says that $f\in I_{D^{m-1}}$. This property is preserved by blowings-up as involved in Step (2). By Lemma \ref{lem:HSlemma}, we also have $\ord(f)=\ord(\Supp D^{(1)})=q$. This property is preserved by blowings-up with smooth centres in $\Supp D_m$ that are normal crossings to $D^{m-1}$; this is the case for the
blowings-up from desingularization of $J$ (see Section \ref{sec:desingularization of J}). Thus
the properties above are preserved by Step (2). 

We can therefore apply Theorem \ref{thm:char} to conclude that $(X,D,E,Y)$ is stable-snc at every point of $\Sigma_{\Omega,q}$, for $(\Omega,q)\in K(X,D)$.
\end{proof}

\section{The non-reduced case}\label{sec:fixingmultiplicities}
The previous sections establish Theorem \ref{thm:maintriples} in the case that $D$ is reduced. In this section we describe the blowings-up necessary to establish the non-reduced case.
In other words, we assume that $(X,D_{\text{red}},E)$ is stable-snc, and we 
prove Theorem \ref{thm:maintriples} under this assumption. 

The algorithm is a simple modification of that in \cite[Section 8]{BV}, to account for the
fact that the components of $X$ and therefore of $D$ are not necessarily of the same
dimension here. For this reason, we only give the modified algorithm and refer to \cite{BV}
for the proof.

We define an equivalence relation on the components of $D$ at a point of $X$.

\begin{definition}\label{def:equivcompofD}
Let $a\in X$ and let $D_1$, $D_2$ denote components of $D$ at $a$. Assume that, for
each $i=1,2$, $D_i \subset X^{(i)}$, where $X^{(i)}$ is a component of $X$ of codimension $c_i$ 
in a minimal local embedding variety $Z$ of $X$ at $a$. We say that $D_1$ and $D_2$ are \emph{equivalent} (\emph{at} $a$) if either $D_1=D_2$ or the irreducible component of $D_1\cap D_2$ at $a$ has codimension $c_1+c_2+1$ in $Z$.
\end{definition}

Given $a\in X$, let $\kappa_X(a)$ denote the number of components of $X$ at $a$, and let $q(a)$ 
denote the number of equivalence classes present in the set of components of $D$ at $a$. 
Define $\iota:X\rightarrow\mathbb{N}^2$ by $\iota(a):=(\kappa_X(a),q(a))$. We give $\mathbb{N}^2$ the partial order where $(\kappa_1,q_1)\geq(\kappa_2,q_2)$ means that 
$\kappa_1\geq \kappa_2$ and $q_1\geq q_2$. Then
$\iota$ is upper semi-continuous. Therefore, the maximal locus of $\iota$ is a closed set.

Each irreducible component $Q$
of the maximal locus of $\iota$ consists of only stable-snc points or only non-stable-snc points, because all points of $Q$ belong to the same irreducible components of
$D$. We blow up with center $C =$ the union of the components of the maximal locus of
$\iota$ that contain only non-stable-snc points. In the preimage of $C$, $\iota$ decreases. 

Let $W$ be the union of the components of the maximal locus consisting of stable-snc points. The 
blowing-up above is an isomorphism on $W$, so $(X',D')$ is stable-snc on $W' = W$, and 
therefore in a neighbourhood of $W'$. For this reason, the union of the components of the maximal locus of $\iota$ on $X'\setminus W'$ that contain only
non-stable-snc points, is closed in $X'$. Therefore, we can repeat the procedure on $X'\setminus W'$.

Clearly, $\mathbb{N}^2$ has no infinite decreasing sequences with respect to the
order above. After the blowing-up above, the maximal values of $\iota$ on the non-stable-snc locus
of $(X,D)$ decrease. Therefore, after a finite number of iterations of the
procedure above, the non-stable-snc locus becomes empty.

\begin{remark}
Suppose that $(X,D_{\text{red}})$ is stable-snc. Then the blowing-up sequence in this section is given simply by the
desingularization algorithm for
$\Supp D$, but blowing up only those components of the maximal locus of the invariant on the non-stable-snc locus. 
\end{remark}

\bibliographystyle{alpha}

\end{document}